\documentclass[11pt,preprint]{imsart}
\setattribute{journal}{name}{}
\RequirePackage[OT1]{fontenc}
\RequirePackage{amsthm,amsmath,amssymb,amsfonts,csquotes,mathrsfs}\RequirePackage{graphicx}
\RequirePackage{xr}
\RequirePackage[numbers]{natbib}
\RequirePackage[colorlinks,citecolor=blue,urlcolor=blue]{hyperref}
\RequirePackage{cleveref}
\RequirePackage{bbm}
\RequirePackage{epstopdf}
\RequirePackage{tikz}
\RequirePackage{subcaption}
\DeclareMathOperator\arctanh{arctanh}
\numberwithin{equation}{section}
\theoremstyle{plain}
\newtheorem{thm}{Theorem}[section]
\newtheorem{remark}[thm]{Remark}
\newtheorem{defn}[thm]{Definition}
\newtheorem{ppn}[thm]{Proposition}
\newtheorem{cor}[thm]{Corollary}
\newtheorem{lem}[thm]{Lemma}

\newcommand{\B}{\mathbb{B}}

\newcommand{\E}{\mathbb{E}}

	\renewcommand{\P}{\mathbb{P}}
\newcommand{\Q}{\mathbb{Q}}
\newcommand{\R}{\mathbb{R}}

\newcommand{\cB}{\mathcal{B}}

\newcommand{\cD}{\mathcal{D}}

\newcommand{\cG}{\mathcal{G}}

\newcommand{\cP}{\mathcal{P}}

\newcommand{\cR}{\mathcal{R}}

\def\beq{\begin{equation}}
\def\eeq{\end{equation}}
\def\ba{\begin{enumerate}[(a)]}
\def\bei{\begin{enumerate}[(i)]}
\def\be{\begin{enumerate}[(1)]}
\def\ee{\end{enumerate}}
\def\bi{\begin{itemize}}
\def\ei{\end{itemize}}
\def\beg{\begin{eg}}
\def\eeg{\end{eg}}
\def\bd{\begin{defn}}
\def\ed{\end{defn}}
\def\bt{\begin{thm}}
\def\et{\end{thm}}
\def\bl{\begin{lemma}}
\def\el{\end{lemma}}
\def\bfac{\begin{fact}}
\def\efac{\end{fact}}

\def\bc{\begin{cor}}
\def\ec{\end{cor}}
\def\bp{\begin{prop}}
\def\ep{\end{prop}}
\def\bo{\begin{observe}}
\def\eo{\end{observe}}
\def\RR{\mathbb{R}}

\def\beg{\begin{eg}}
\def\eeg{\end{eg}}

\begin{document}


\begin{frontmatter}
\title{Joint estimation of parameters in Ising model}
\runtitle{Joint estimation of parameters in Ising model}

	\thankstext{m1}{Partially supported  by NSF Grant DMS-1664650.}
	\thankstext{m2}{Partially supported by NSF Grant DMS-1712037.}

\begin{aug}

\author{\fnms{Promit} \snm{Ghosal}\thanksref{m1}\ead[label=e1]{pg2475@columbia.edu}}
		\and
		\author{\fnms{Sumit} \snm{Mukherjee}\thanksref{m2}\ead[label=e2]{sm3949@columbia.edu}}
		
\runauthor{Ghosal \& Mukherjee}		
			
%
%
\affiliation{Columbia University}

\address{Department of Statistics\\
			1255 Amsterdam Avenue\\
			New York, NY-10027. \\
			\printead{e1}\\
			\printead{e2}}
			

\end{aug}

%
%
%
%
%
%
%

\begin{abstract}

We study joint estimation of the inverse temperature and magnetization parameters $(\beta,B)$ of an Ising model with a non-negative coupling matrix $A_n$ of size $n\times n$, given one sample from the Ising model. We give a general bound on the rate of consistency of the bi-variate pseudo-likelihood estimator. Using this, we show that estimation at rate $n^{-1/2}$ is always possible if $A_n$ is the adjacency matrix of a bounded degree graph. If $A_n$ is the scaled adjacency matrix of a graph whose average degree goes to $+\infty$, the situation is a bit more delicate. In this case estimation at rate $n^{-1/2}$ is still possible if the graph is not regular (in an asymptotic sense). Finally,  we show that consistent estimation of both parameters is impossible if the graph is Erd\"os-Renyi with parameter $p>0$ free of $n$, thus confirming that estimation is harder on approximately regular graphs with large degree.

\end{abstract}



\end{frontmatter}

\section{Introduction}\label{sec:intro}

Suppose $\beta>0,B\ne 0$ are unknown parameters, and $A_n$ is an $n\times n$ symmetric matrix with non-negative entries with $0$ on the diagonal. For ${\bf x}:=(x_1,\cdots,x_n)\in \{-1,1\}^n$, define a p.m.f. $\P_{n,\beta,B}(.)$ by setting
\begin{align}\label{eq:ising}
\P_{n,\beta,B}({\bf X}={\bf x})=\frac{1}{Z_n(\beta,B)}e^{\frac{\beta}{2}{\bf x}'A_n{\bf x}+B\sum_{i=1}^nx_i}.
\end{align}

This is the Ising model with coupling matrix $A_n$, and inverse temperature parameter $\beta$ and magnetization parameter $B$. Study of Ising models is a growing area which has received significant attention in Statistics and Machine Learning in recent years. The theoretical investigation into the properties of Ising models can be broadly classified into two categories. One of the branches assumes that the matrix $A_n$ is the unknown parameter of interest, and focuses on estimating $A_n$ under the assumption that i.i.d. copies ${\bf X}^{(1)},\cdots,{\bf X}^{(p)}$ are available from the model described in \ref{eq:ising} (c.f. \cite{structure_learning,bresler,highdim_ising,ising_nonconcave} and references there-in).
 Another branch works under the assumption that only one observation ${\bf X}$ is available from the model in \eqref{eq:ising} (c.f. \cite{BhattacharyaM,Chatterjee,comets_mle,gidas,guyon} and references there-in). In this setting, estimation of the whole matrix $A_n$ (which has $n^2$ entries) is impossible from a vector ${\bf X}$ of size $n$. As such, the standard assumption is that the matrix $A_n$ is completely specified, and the focus is on estimating the parameters $(\beta,B)$. In this direction, the behavior of the MLE for the Curie-Weiss model (when $A_n$ is the scaled adjacency matrix of the complete graph) was studied in \cite{comets_mle}, where the authors showed that in the regime $\beta>0,B\ne 0$, the MLE of $\beta$ is $\sqrt{n}$ consistent for $\beta$ if $B$ is known, and vice versa. They also show that if both $(\beta,B)$ are unknown, then the joint MLE for the model does not exist with probability 1. This raises the natural question as to whether there are other estimators which work in this case. 
Focusing on the case when $B=0$ is known, \cite{Chatterjee} gave general sufficient conditions under which the pseudo-likelihood estimate for $\beta$ is $\sqrt{n}$ consistent. Developing on this approach, \cite{BhattacharyaM} studies the behavior of the rate of consistency of the pseudo-likelihood estimator at all values of $\beta$, demonstrating interesting phase transition properties in the rate of the pseudo-likelihood estimator. 
The question of joint estimation of $(\beta,B)$ for a general matrix $A_n$ was raised in \cite{Chatterjee}. Up to the best of the our knowledge, this question has not been addressed in the literature. This will be the focus of this paper.

\subsection{Main results}

Throughout this paper we will assume that  $(\beta,B)$ are unknown parameters of interest, and the coupling matrix $A_n$ has non-negative entries and is completely known. We will also assume the following two conditions
\begin{align}\label{eq:bd}
\max_{i\in [n]}\sum_{j=1}^nA_n(i,j)\le &\gamma,\\
\label{eq:non_trivial}
\liminf_{n\rightarrow\infty}\frac{1}{n}\sum_{i,j=1}^nA_n(i,j)&>0.
\end{align}
Here  $\gamma$ is a finite constant free of $n$.
 Note that \eqref{eq:bd} implies that $A_n$ satisfies \cite[Eqn  (1.10)]{BasakM}, as well as $||A_n||_2\le \gamma$ (\cite[Cond (a),Thm 1.1]{Chatterjee}), where $||A_n||_2$ is the operator norm of $A_n$. 
The most common examples of the coupling matrix $A_n$ are scaled adjacency matrices of labelled simple graphs, defined as follows:
\begin{defn}\label{def:graph}
For a graph $G_n$ with vertices labelled by $[n]:=\{1,2,\cdots,n\}$, define the coupling matrix $A_n$ by setting
$$A_n(i,j):=\frac{n}{2E(G_n)}{1}\{\text{vertices }i\text{ and }j\text{ are connected in }G_n\},$$
where $E(G_n)$ is the number of edges in the graph $G_n$.
\end{defn}
The scaling of the adjacency matrix by $\frac{n}{2E(G_n)}$ ensures that the resulting Ising model has non-trivial phase transition properties (see for e.g. \cite{BasakM}), which is of much interest in Statistical Physics and Applied Probability. The influence of phase transition on Inference has received recent attention (c.f. \cite{BhattacharyaM,MMY}). Under this scaling, \eqref{eq:non_trivial} holds trivially, as $\sum_{i,j=1}^nA_n(i,j)=n$. Condition \eqref{eq:bd} demands that the maximum degree of $G_n$ is of the same order as the average degree. Below we give examples of some graphs for which \eqref{eq:bd} holds. We will also use these  as running examples for all our future assumptions.

\begin{defn}
For a simple graph $G_n$ on $n$ vertices, let $(d_1(G_n),\cdots,d_n(G_n))$ denote the labelled degrees of $G_n$.
\end{defn}

\begin{enumerate}
\item[(a)]
$G_n$ is a $d_n$ regular graph for some $d_n\in [1,n-1]$. In this case we have $\sum_{j=1}^nA_n(i,j)=1$ for all $i\in [n]$, and so  \eqref{eq:bd} holds with $\gamma=1$.
\\

\item[(b)]
$G_n$ is an Erdos-Renyi graph with parameter $p_n\ge (1+\delta)\frac{\log n}{n}$, where $\delta>0$ is fixed. In this case we have
$$\frac{\max_{i\in [n]}d_i(G_n)}{np_n}\stackrel{p}{\rightarrow }1,\quad \frac{2E(G_n)}{n^2p_n}\stackrel{p}{\rightarrow}1$$
and so \eqref{eq:bd} holds with probability tending to $1$ for $\gamma=2$.
\\

\item[(c)]
$G_n$ is a convergent sequence of dense graph converging to the graphon $W$ which is not $0$ (see \cite{lovasz_book} for a survey on the literature on graphons/graph limits). In this case we have $$\max_{i\in [n]}d_i(G_n)\le n-1,\quad \frac{2 E(G_n)}{n^2}\rightarrow \int_{[0,1]^2}W(x,y)dxdy,$$ and so \eqref{eq:bd} holds for all large $n$ with $\gamma=\frac{2}{\int_{[0,1]^2}W(x,y)dxdy}$.
\\

\item[(d)]
$G_n$ is a bi-regular bipartite graph with parameters $(a_n,b_n,c_n,d_n,p)$ defined as follows:

\begin{defn}
$G_n$ has bipartition sets $G_{n,1}$ and $G_{n,2}$, with sizes $a_n$ and $b_n$ respectively, and each vertex in $G_{n1}$ has degree $c_n$, and each vertex in $G_{n2}$ has degree $d_n$. 
 Finally assume that $$\lim_{n\rightarrow\infty}\frac{a_n}{n}=p\in (0,1),$$ 
Note that the parameters $(a_n,b_n,c_n,d_n)$ are related, as we have $E(G_n)=a_nc_n=b_nd_n$, and $a_n+b_n=n$.
\end{defn}

In this case we have $d_i(G_n)\le \max(c_n,d_n)$ and $E(G_n)= a_nc_n$, and so  \eqref{eq:bd} holds for all large $n$ with $\gamma=\frac{1}{\max(p,1-p)}$. 

\end{enumerate}

We will now introduce the bivariate pseudo-likelihood estimator.
\begin{defn}
For any $i\in [n]$ we have 
\begin{equation*}
\P_{n,\beta,B}(X_i=1|X_j,j\ne i)=\frac{e^{\beta m_i({\bf x})+B}}{e^{\beta m_i({\bf x})+B}+e^{-\beta m_i({\bf x})-B}},
\end{equation*}
where $m_i({\bf x}):=\sum_{j=1}^nA_n(i,j)x_j$. Define the pseudo-likelihood as the product of the one dimensional conditional distributions: (see \cite{B1,B2})
$$\prod_{i=1}^n \P_{n,\beta,B}(X_i=x_i|X_j,j\ne i)=2^{-n}\text{ exp }\Big\{\sum_{i=1}^n\Big(\beta x_i m_i({\bf x})+Bx_i-\log\cosh(\beta m_i({\bf x})+B)\Big)\Big\}$$
On taking $\log$ and differentiating this with respect to $(\beta,B)$ we get the vector $(Q_n(\beta,B|{\bf x}),R_n(\beta,B|{\bf x}))$, where
\begin{align*}
Q_n(\beta,B|{\bf x}):=&\sum_{i=1}^nm_i({\bf x})(x_i-\tanh(\beta m_i({\bf x})+B)),\\
R_n(\beta,B|{\bf x}):=&\sum_{i=1}^n(x_i-\tanh(\beta m_i({\bf x})+B)).
\end{align*}
The bi-variate equation  $$PL_n(\beta,B|{\bf x}):=(Q_n(\beta,B|{\bf x}),R_n(\beta,B|{\bf x}))=(0,0)$$ will be referred to as the pseudo-likelihood equation in this paper. If the pseudo-likelihood equation has a unique root in $(\beta,B)\in \R^2$, denote it by $(\hat{\beta}_n,\hat{B}_n)$.  This is the pseudo-likelihood estimator for the parameter vector $(\beta,B)$. 
\end{defn}

%
%
%
%

\begin{defn}\label{def:bi}
Suppose  $U_n$ and $V_n$ are two non-negative random variables on the probability space $\Big(\{-1,1\}^n,\P_{n,\beta,B}\Big),$ where $\P_{n,\beta,B}$ is the Ising p.m.f. given in \eqref{eq:ising}. 
We will say $U_n=O_p(V_n)$ if the sequence $\frac{U_n}{V_n}$ is tight. In particular this implies that $\lim_{n\rightarrow\infty}\P_{n,\beta,B}(U_n>0,V_n=0)=0.$
We will say $U_n=\Theta_p(V_n)$, if both $U_n=O_p(V_n)$ and $V_n=O_p(U_n)$. We will say $U_n=o_p(V_n)$ if $\frac{U_n}{V_n}\stackrel{p}{\rightarrow}0$.
\end{defn}

\begin{defn}
Let $\Theta\subset\R^2$ denote the set of all parameters $(\beta,B)$ such that $\beta>0,B\ne 0$.
\end{defn}

Our first result gives a general upper bound on the error of the pseudo-likelihood estimator.

\begin{thm}\label{thm:pl}

Suppose ${\bf X}=(X_1,\cdots,X_n)$ is an observation from the Ising model \eqref{eq:ising}, where the coupling matrix $A_n$ satisfies \eqref{eq:bd} and \eqref{eq:non_trivial}, and $(\beta,B)\in \Theta$. Set $$T_{n}({\bf x}):=\frac{1}{n}\sum_{i=1}^n\Big(m_i({\bf x})-\bar{m}({\bf x})\Big)^2.$$
\begin{enumerate}
\item[(a)]

The pseudo-likelihood estimator $(\hat{\beta}_n,\hat{B}_n)$ exists iff
${\bf x}\in A_{1,n}^c\cap A_{2,n}^c\cap A_{3,n}^c\cap A_{4,n}^c$, where
 \begin{align*}
 A_{1,n}:=&\{{\bf x}\in \{-1,1\}^n:T_n({\bf x})=0\},\\
 A_{2,n}:=&\{{\bf x}\in \{-1,1\}^n:m_i({\bf x})x_i=|m_i({\bf x})|\text{ for all }i\in [n]\},\\
 A_{3,n}:=&\{{\bf x}\in \{-1,1\}^n:m_i({\bf x})x_i=-|m_i({\bf x})|\text{ for all }i\in [n]\},\\
 A_{4,n}:=&\{{\bf 1},-{\bf 1}\}.
 \end{align*}
\item[(b)]
If the true parameter is $(\beta_0,B_0)\in\Theta$, then we have 
$$\lim_{n\rightarrow\infty}\P_{n,\beta_0,B_0}(A_{2,n}^c\cap A_{3,n}^c\cap A_{4,n}^c)=1.$$

\item[(c)] 
Further, if $\frac{1}{T_n({\bf X})}=o_p(\sqrt{n})$, then 
$$||\hat{\beta}_n-\beta_0,\hat{B}_n-B_0||=O_p\Big(\frac{1}{\sqrt{n}T_n({\bf X})}\Big).$$ In particular, $(\hat{\beta}_n,\hat{B}_n)$ is jointly consistent for $(\beta,B)$ in $\Theta$.

\end{enumerate}
\end{thm}

An immediate corollary of  Theorem \ref{thm:pl} is the following corollary.

\begin{cor}\label{cor:rootn}
In the setting of Theorem \ref{thm:pl}, if we further have
\begin{align}\label{eq:rootn}
T_n({\bf X})=\Theta_p(1),
\end{align} then
$||\hat{\beta}_n-\beta_0,\hat{B}_n-B_0||=O_p\Big(\frac{1}{\sqrt{n}}\Big)$ under $\P_{n,\beta_0,B_0}$, i.e. the pseudo-likelihood estimator is jointly $\sqrt{n}$ consistent.

\end{cor}

Corollary \ref{cor:rootn} shows that \eqref{eq:rootn} is a  sufficient condition for $\sqrt{n}$ consistency of the pseudo-likelihood estimate. Note that condition \eqref{eq:rootn} is an implicit condition, and it is not clear when this will hold. We will now give an exact characterization for \eqref{eq:rootn} in terms of the matrix $A_n$ for \enquote{mean field} matrices, introduced in the following definition.

\begin{defn}
We say that a sequence of matrices $\{A_n\}_{n\ge 1}$ satisfies the mean field condition, if we have
\begin{align}\label{eq:mean_field}
\lim_{n\rightarrow\infty}\frac{1}{n}\sum_{i,j=1}^nA_n(i,j)^2=0.
\end{align}
\end{defn}
Condition \ref{eq:mean_field} was first introduced in \cite{BasakM} to study the limiting behavior of normalizing constant of Ising and Potts models. In particular, if $A_n(i,j)=\frac{n}{2E(G_n)}G_n(i,j)$, where $G_n$ is the adjacency matrix of a graph, then \eqref{eq:mean_field} holds iff $\E(G_n)\gg n$. Indeed, this is because
$$\sum_{i,j=1}^n A_n(i,j)^2=\frac{n^2}{4E(G_n)^2}\sum_{i,j=1}^nG_n(i,j)^2=\frac{n^2}{4E(G_n)},$$
which is $o(n)$ iff $\E(G_n)\gg n$. Thus \eqref{eq:mean_field} holds in the following examples:

\begin{enumerate}
\item[(a)]
$G_n$ is a $d_n$ regular graph with $\lim_{n\rightarrow\infty}d_n=\infty$. In this case we have
$2E(G_n)=nd_n\gg n$.

\item[(b)]
$G_n$ is an Erdos-Renyi graph with parameter $p_n\gg \frac{1}{n}$. In this case we have
$E(G_n)\stackrel{p}{\approx}\frac{n^2p_n}{2}\gg n$.

\item[(c)]
$G_n$ is a convergent sequence of dense graph converging to the graphon $W$ which is not identically $0$. In this case we have $\E(G_n)=\Theta(n^2)$.

\item[(d)]
$G_n$ is a bi-regular bipartite graph with parameters $(a_n,b_n,c_n,d_n,p)$ as in Definition \ref{def:bi}, such that $\lim_{n\rightarrow\infty}(c_n+d_n)=\infty$. In this case we have
$\frac{c_n}{d_n}=\frac{b_n}{a_n}\rightarrow \frac{1-p}{p}$,  and so
$\frac{E(G_n)}{n}=\frac{a_nc_n}{n}\sim p c_n\rightarrow\infty$.

\end{enumerate}

\begin{defn}
Given the coupling matrix $A_n$, let $\cR_n(i):=\sum_{j=1}^nA_n(i,j)$ denote the $i^{th}$ row sum of $A_n$. 
\end{defn}

Our next result now gives a simple sufficient condition for joint $\sqrt{n}$ consistency of the pseudo-likelihood estimator. 
\begin{thm}\label{thm:basak2}
Suppose ${\bf X}=(X_1,\cdots,X_n)$ is an observation from the Ising model \eqref{eq:ising}, where the coupling matrix $A_n$ satisfies \eqref{eq:bd} and \eqref{eq:mean_field}.
If 
\begin{align}\label{eq:irregular}
\liminf_{n\rightarrow\infty}\frac{1}{n}\sum_{i=1}^n(\cR_n(i)-\overline{\cR}_n)^2>0,
\end{align}
 then we have $T_n({\bf X})=\Theta_p(1)$ for all $(\beta,B)\in \Theta$. Consequently we have $(\hat{\beta}_n-\beta)^2+(\hat{B}_n-B)^2=O_p(\frac{1}{n})$.
\end{thm}

Note that \eqref{eq:irregular} and \eqref{eq:bd} together imply \eqref{eq:non_trivial}. Thus if $A_n$ is the scaled adjacency matrix of a graph $G_n$ with $E(G_n)\gg n$, the pseudo-likelihood is $\sqrt{n}$ consistent whenever the graph $G_n$ is slightly irregular. In particular, the pseudo-likelihood is $\sqrt{n}$ consistent in the following examples:

\begin{enumerate}
\item[(c)]
$G_n$ is a convergent sequence of dense graphs converging to the graphon $W$ such that the function $\cR(x):=\int_0^1 W(x,y)dy$ is not constant almost surely Lebesgue measure. In this case we have
$$\lim_{n\rightarrow\infty}\frac{1}{n}(\cR_n(i)-\bar{\cR}_n)^2=\int_0^1 (\cR(x)-\int_0^1 \cR(y)dy)^2 dx>0,$$
and so \eqref{eq:irregular} holds. Also, as previously verified, \eqref{eq:bd} and \eqref{eq:mean_field} holds as well. Thus we have the pseudo-likelihood estimator is $\sqrt{n}$ consistent on $\Theta$.

\item[(d)]
$G_n$ is a bi-regular bipartite graph with parameters $a_n,b_n,c_n,d_n,p$  as defined above, such that
$\lim_{n\rightarrow\infty}(c_n+d_n)=\infty$, and $p\ne \frac{1}{2}$.

In this case \eqref{eq:bd} and \eqref{eq:mean_field} were verified before. Finally, we have
$\cR_n(i)=\frac{n}{a_n}\sim \frac{1}{2p}$ for $i\in G_{n1}$, and $\cR_n(i)\sim \frac{1}{2(1-p)}$ for $i\in G_{n2}$. This gives
$$\frac{1}{n}\sum_{i=1}^n(\cR_n(i)-\bar{\cR}_n)^2\sim \frac{a_n}{n}\Big(\frac{1}{2p}-1\Big)^2+\frac{b_n}{n}\Big(\frac{1}{2(1-p)}-1\Big)^2\rightarrow\frac{(2p-1)^2}{4p(1-p)}>0 .$$ Thus we have the pseudo-likelihood estimator is $\sqrt{n}$ consistent on $\Theta$. Note that if $p=\frac{1}{2}$ the graph $G_n$ is asymptotically regular.

\end{enumerate}

This raises the natural question as to what happens for regular graphs. The following theorem addresses this question by showing that whenever the coupling matrix $A_n$ is mean field and asymptotically regular, the random variable $T_n({\bf X})$ is $o_p(1)$.

\begin{thm}\label{basak}
Suppose ${\bf X}=(X_1,\cdots,X_n)$ is an observation from the Ising model \eqref{eq:ising}, where the coupling matrix $A_n$ satisfies \eqref{eq:bd} and \eqref{eq:mean_field}.
If 
\begin{align}\label{eq:regular}
\lim_{n\rightarrow\infty}\frac{1}{n}\sum_{i=1}^n(\cR_n(i)-\overline{\cR}_n)^2=0,
\end{align}
 then we have $T_n({\bf X})=o_p(1)$ for all $(\beta,B)\in\Theta$.
\end{thm}

Theorem \ref{basak} along with the upper bound of Theorem  \ref{thm:pl} together suggest that $\sqrt{n}$ consistency may not be attained by the pseudo-likelihood estimator for asymptotically regular graphs with degree going to $+\infty$.  The following theorem confirms this conjecture for the special case when $G_n$ is an Erd\"os-Renyi graph with parameter $p>0$, free of $n$. 

\begin{thm}\label{thm:cw}
Suppose ${\bf X}=(X_1,\cdots,X_n)$ is an observation from the Ising model \eqref{eq:ising}, where the coupling matrix is $A_n(i,j)=\frac{1}{(n-1)p}G_n(i,j)$, where $G_n$ is a random graph from $\cG(n,p)$, the Erd\"os-Renyi graph with parameter $p>0$, free of $n$.  Let $t\in (0,1)$ be fixed, and let $$\Theta_t:=\{(\theta,\beta)\in (0,\infty)^2: t=\tanh(\beta t+B)\}.$$ Let $\P^{\mathrm{er}}_{n,\beta,B}$ denote the joint law of $({\bf X})$ and $G_n$ on $\{-1,1\}^n\times \{0,1\}^{n\choose 2}$. Then, setting $\Q_n$ to be product measure on $\{-1,1\}^n$ under which $\Q_n(X_i=1)=\frac{1}{1+e^{-2t}}=1-\Q_n(X_i=-1)$, we have that $\Q_n\times \cG(n,p)$ is contiguous to $\P_{n,\beta,B}^{\mathrm{er}}$ for every $(\beta,B)\in \Theta_t$.
Consequently, under $\P_{n,\beta,\B}^{\mathrm{er}}$ there does not exist any sequence of estimates (functions of $({\bf X},G_n)$) which is consistent for $(\beta,B)$ in $\Theta_t$ (and hence in $\Theta$).
\end{thm}
\begin{remark}
It was pointed out in \cite{comets_mle} that the MLE for $(\beta,B)$ doesn't exist for the Curie Weiss model. The above Theorem extends this by showing that consistent estimates do not exist when the underlying graph is Erd\"os-Renyi. Note that if we set $p=1$ in the Erd\"os-Renyi model we get a complete graph on $n$ vertices, which corresponds to the Curie-Weiss model. We conjecture that there are no $\sqrt{n}$ consistent estimates for both parameters whenever the graph sequence is regular with degree going to $+\infty$.
\end{remark}
%
%

If the average degree of a graph sequence does not go to $+\infty$, joint estimation of both parameters at rate $\sqrt{n}$ is always possible, as shown in the following theorem.

\begin{thm}\label{ppn:z1}
Suppose ${\bf X}=(X_1,\cdots,X_n)$ is an observation from the Ising model \eqref{eq:ising}, where the coupling matrix $A_n$ satisfies \eqref{eq:bd} and
\begin{align}\label{eq:Sparseness}
\liminf_{n\to \infty}\frac{1}{n}\sum_{i,j=1}^n A^2_n(i,j) >0. 
\end{align}
 Then, for any  $(\beta_0,B_0)\in \Theta$ we have
$\sqrt{n}(\hat{\beta}_n-\beta_0,\hat{B}_n-B_0)=O_p(1)$, i.e. the joint pseudo-likelihood estimator is $\sqrt{n}$ consistent.

\end{thm}

Note that \eqref{eq:bd} and \eqref{eq:Sparseness} together imply \eqref{eq:non_trivial}. To see how \eqref{eq:Sparseness} captures sparse graphs, recall that for any graph with adjacency matrix $G_n$ and $A_n(i,j):=\frac{n}{2E(G_n)}G_n(i,j)$ we have
$\sum_{i,j=1}^nA_n(i,j)^2=\frac{n^2}{4 E(G_n)}.$
Thus if \eqref{eq:Sparseness} holds, then we must have $E(G_n)=O(n) $, which says that the graph sequence is sparse. In particular Theorem \ref{ppn:z1} shows $\sqrt{n}$ consistency when the underlying graph $G_n$ has a uniformly bounded degree sequence,  irrespective of whether $G_n$ is regular or not. 
\\

To complete the picture, we show that if one of the two parameters are known, then the pseudo-likelihood estimator for the other parameter is $\sqrt{n}$ consistent, for all $(\beta,B)\in \Theta$. Thus joint estimation is indeed a much harder problem than estimation of the individual parameters.  The proof of this proposition appears in the appendix.

\begin{ppn}\label{ppn:rootn}

Suppose ${\bf X}=(X_1,\cdots,X_n)$ is an observation from the Ising model \eqref{eq:ising}, where the coupling matrix $A_n$ satisfies \eqref{eq:bd} and \eqref{eq:non_trivial}.
\begin{enumerate}
\item[(a)]
If $B$ is known, then the equation $Q_n(\beta,B|{\bf x})=0$ has a unique root $\hat{\beta}_n$ which satisfies $\sqrt{n}(\hat{\beta}_n-\beta_0)=O_p(1)$ under $\P_{n,\beta_0,B}$.

\item[(b)]
If $\beta$ is known, then the equation $R_n(\beta,B|{\bf x})=0$ has a unique root $\hat{B}_n$ which satisfies $\sqrt{n}(\hat{B}_n-B_0)=O_p(1)$ under $\P_{n,\beta,B_0}$.

\end{enumerate}
\end{ppn}

\subsection{Interpretation of results for graphs}

Even though all our results apply for general matrices with non-negative entries, the most interesting examples for our theorems are the cases when $A_n$ is the scaled adjacency matrix of a simple graph $G_n$ as in Definition \ref{def:graph}. Also the conditions take a simpler form. This subsection describes all our results in this special case. Recall that $(d_1(G_n),\cdots,d_n(G_n))$ are the degrees of $G_n$, and let $\bar{d}(G_n)$ denote the average degree. Also assume that $(\beta,B)\in \Theta$, as has been the case throughout the paper. Finally note that \eqref{eq:bd} is equivalent to $\max_{i\in [n]}d_i(G_n)=O(\bar{d}(G_n))$, which we will assume throughout this subsection.

\begin{itemize}
\item
For any graph $G_n$, the pseudo-likelihood estimate of $\beta$ is $\sqrt{n}$ consistent if $B$ is known, and vice versa (Proposition \ref{ppn:rootn}).
\\

\item
If $E(G_n)\gg n$, and the graph is somewhat irregular as captured by the condition $$\liminf_{n\rightarrow\infty}\frac{1}{n}\sum_{i=1}^n\Big[\frac{d_i(G_n)}{\bar{d}(G_n)}-1\Big]^2>0,$$
then the pseudo-likelihood estimator for $(\beta,B)$ is jointly $\sqrt{n}$ consistent (Theorem \ref{thm:basak2}).
\\

\item
If $E(G_n)\gg n$ and the graph is somewhat  regular as captured by the condition $$\limsup_{n\rightarrow\infty}\frac{1}{n}\sum_{i=1}^n\Big[\frac{d_i(G_n)}{\bar{d}(G_n)}-1\Big]^2=0,$$
then we believe that the pseudo-likelihood estimator for $(\beta,B)$ is not jointly $\sqrt{n}$ consistent (Theorem \ref{basak}). The only reason this statement is suggestive and not rigorous is that Theorem \ref{thm:pl} only provides an upper bound and not a matching lower bound.
\\

\item
For the particular case when $G_n$ is an Erd\"os-Renyi random graph with parameter $p>0$ free of $n$, there are no estimators which are jointly consistent for $(\beta,B)$ (Theorem \ref{thm:cw}). Thus indeed the estimation problem is harder on asymptotically regular graphs with large degree.
\\

\item
If $G_n$ is a graph with $\max_{i\in [n]}d_i(G_n)$ bounded, then the pseudo-likelihood estimator for $(\beta,B)$ is jointly $\sqrt{n}$ consistent (Theorem \ref{ppn:z1}), irrespective of whether $G_n$ is regular or not.

\end{itemize}
 Figure~\ref{fig:SummaryTree} gives a gist of the above discussion on a summary tree. 
\begin{figure}
\begin{tikzpicture}[sibling distance=20em,
  every node/.style = {shape=rectangle, rounded corners,
    draw, align=center,
    top color=white, bottom color=green!20}]]
  \node {Graph $G_n$,\\Avg. Deg. $\bar{d}(G_n)$}
    child { node {$\bar{d}(G_n)$ bounded} 
      child{ node{Estimation at \\$\sqrt{n}$-rate \textbf{possible}.\\See Theorem~\ref{ppn:z1}.}}}
    child { node {$\bar{d}(G_n)$ unbounded}
      child { node {Regular}
        child { node {Estimation at \\$\sqrt{n}$-rate may be\\\textbf{impossible}. See Theorem~\ref{basak}\\ and Theorem \ref{thm:cw}.} }}
      child { node {Irregular} 
        child { node {Estimation at\\$\sqrt{n}$-rate \textbf{possible}.\\See Theorem~\ref{thm:basak2}.} } }};
\end{tikzpicture}
\caption{Summary tree of our results.}
\label{fig:SummaryTree}
\end{figure}
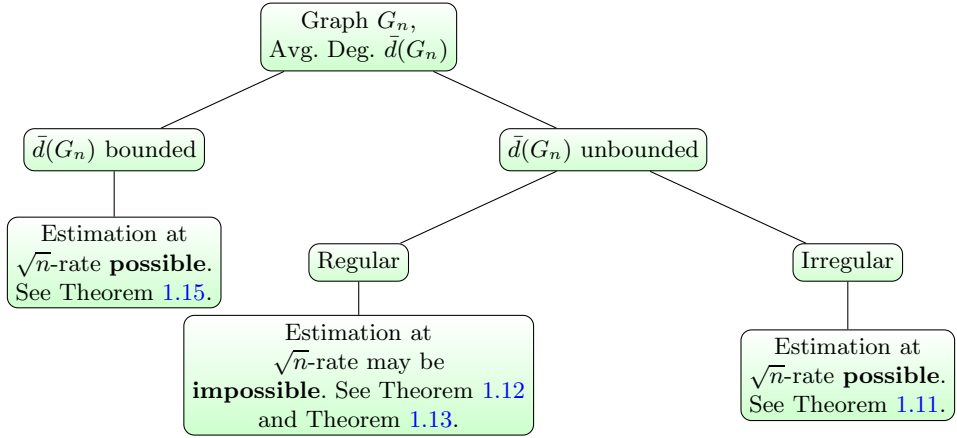 
\\

\subsection{Simulation}
Our results demonstrate a dichotomy in the joint $\sqrt{n}$-consistency of $(\beta,B)$ based on whether the coupling matrix is approximately regular or not.
In what follows, we address this dichotomy using simulation. At first, we fix $30$ different values of the pair $(\beta, B)$ on the line $m= \tanh(m\beta +B)$ for $m=0.3$. Next, we draw two random $d$-regular graphs $G_1$ and $G_2$ with $d=4$ and $d=50$, with $n=100$ nodes. For each value of $(\beta, B)$, we generate a sample from the Ising model with scaled adjacency matrices for the graphs $G_1$ and $G_2$. 
On each of those $30$ different samples, we estimate $(\beta, B)$ by solving the bivariate pseudolikelihood equation. We repeat the same experiment with number of nodes $n=200$, and random $d$-regular graphs ($d=4,50$). In Figure~\ref{fig:TOF}, we plot the corresponding pseudolikelihood estimates of $(\beta, B)$ for $n=100$ and $n=200$ respectively. In both the figures, plots of the estimates for the case $d=4$ (resp. $d=50$) are colored in green (resp. red). Notice that the fit in the case when $d=4$ is more prominent in comparison to the case when $d=50$.

\begin{figure}
\centering
\begin{subfigure}{7cm}       
\includegraphics[width=8.5cm,height=7cm]{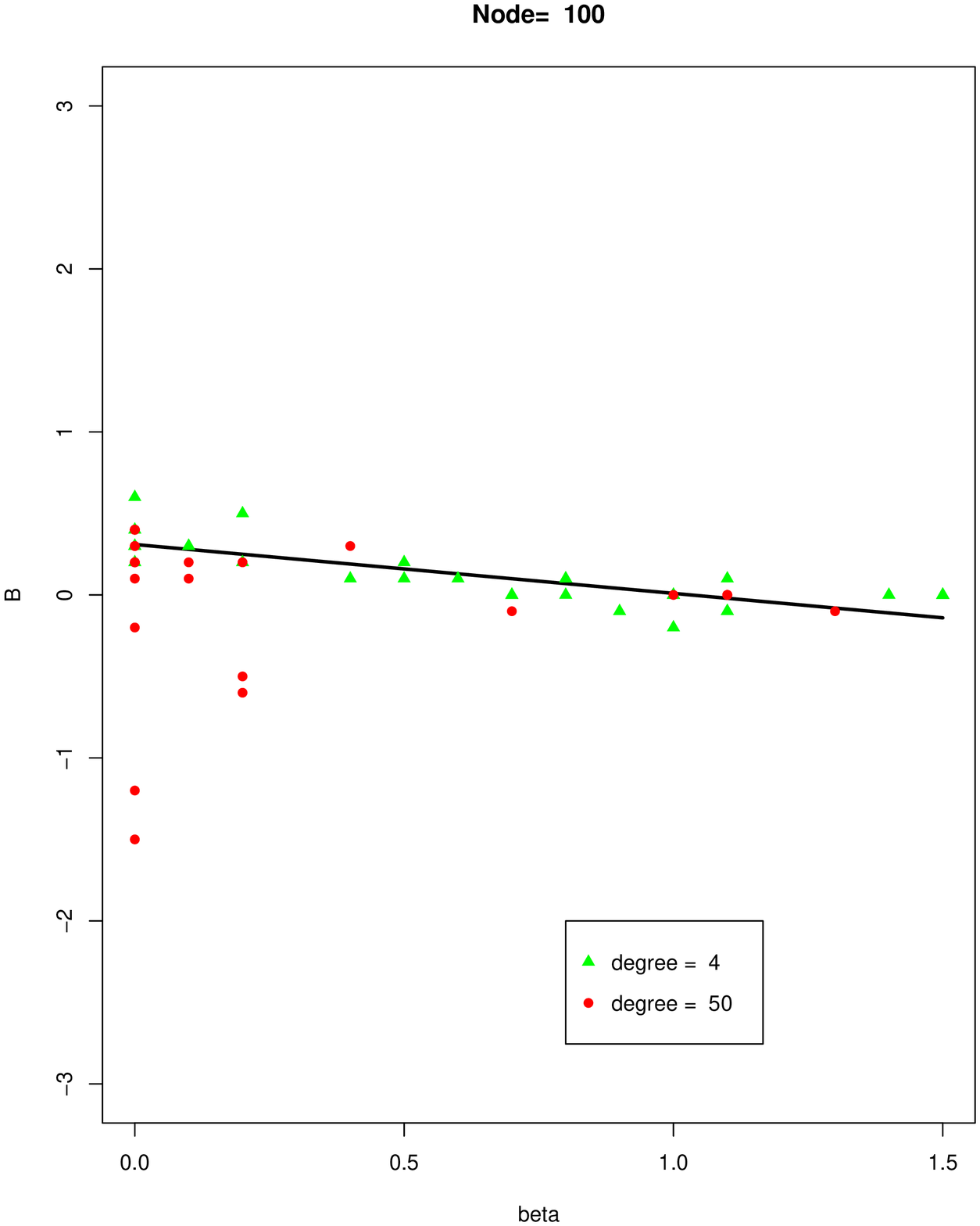}
\label{fig:100Exp}
\end{subfigure}
\hspace{0.5cm}
\begin{subfigure}{7cm}
\centering
\includegraphics[width=8.5cm,height=7cm]{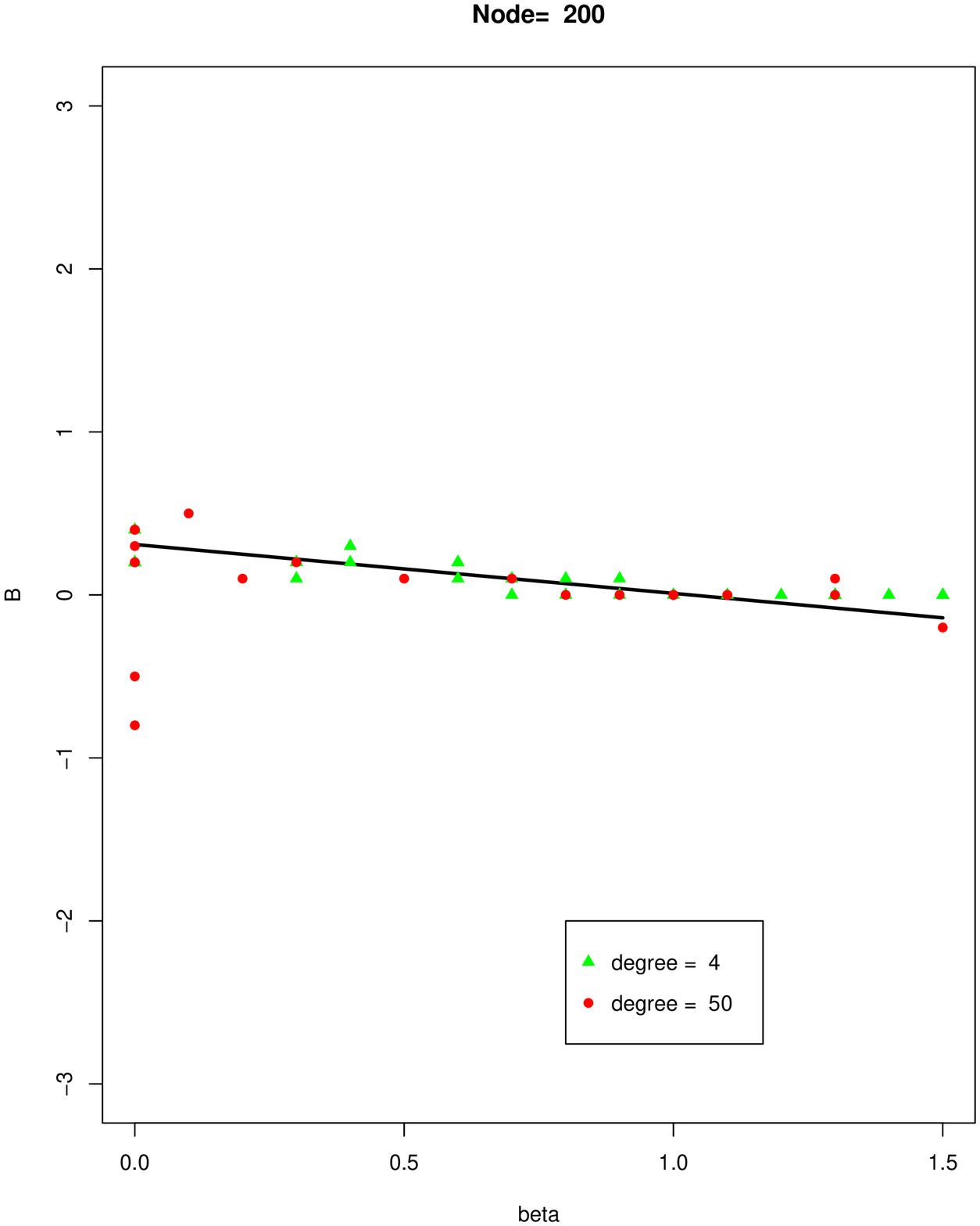}
\label{fig:200Exp}
\end{subfigure}
\caption{Plot of the pseudo-likelihood estimate $(\hat{\beta}_n,\hat{B}_n)$ for $n=100$ (on left) and $n=200$ (on right) respectively where $(\beta,B)$ lie on the line $m=\tanh(m\beta+B)$ (black line in the plot) for $m=0.3$.}\label{fig:TOF}
\end{figure}

The rest of the paper is outlined as follows: Section \ref{sec:two} details the proof of Theorem \ref{thm:pl}. Section \ref{sec:three} proves Theorem \ref{thm:basak2} and Theorem \ref{basak} with the help of Theorem \ref{lem:basak}, the proof of which is deferred to the appendix. Finally, section \ref{sec:four} gives the proof of Theorem \ref{thm:cw} and Theorem~\ref{ppn:z1}. The proof of Proposition \ref{ppn:rootn} is also deferred to the appendix.

\section*{Acknowledgement}
We thankfully acknowledge helpful discussions with Bhaswar B. Bhattyacharya and Sourav Chatterjee  at various stages of this work.

\section{Proof of Theorem \ref{thm:pl} }\label{sec:two}

The following Lemma is a collection of estimates to be used throughout the rest of this paper.

\begin{lem}\label{lem:basak2}
Suppose ${\bf X}=(X_1,\cdots,X_n)$ is an observation from the Ising model \eqref{eq:ising}, where the coupling matrix $A_n$ satisfies \eqref{eq:bd} and \eqref{eq:mean_field}.

 Then setting $$f_n({\bf x}):=\frac{\beta}{2}{\bf x}'A_n{\bf x}+B\sum_{i=1}^nx_i $$ and $$b_i({\bf x}):=\E(X_i|X_j=x_j,j\ne i)=\tanh(\beta m_i({\bf x})+B),$$ the following hold:
\begin{align}
\label{eq:bas3}&\limsup_{n\rightarrow\infty}\frac{1}{n^2}\E \Big[f_n({\bf X})-f_n(b({\bf X}))]^2=0,\\
\label{eq:bas2}&\limsup_{n\rightarrow\infty}\frac{1}{n}\E \Big[\sum_{i=1}^n(X_i-b_i({\bf X}))m_i({\bf X})\Big]^2<\infty,\\
\label{eq:bas1}&\limsup_{n\rightarrow\infty}\frac{1}{n}\E \Big[\sum_{i=1}^n(X_i-b_i({\bf X}))\Big]^2<\infty.
\end{align}

\end{lem}

\begin{proof}[Proof of Lemma \ref{lem:basak2}]

 Various versions of these estimates exist already in the literature. 
 In particular, \eqref{eq:bas3} follows on invoking \cite[Lemma 3.1]{CD}  or \cite[Lemma 3.2]{BasakM} along with the assumption that $A_n$ satisfies \eqref{eq:mean_field}, and  \eqref{eq:bas2} follows on invoking \cite[Lemma 3.2]{CD} along with the assumption that $A_n$  satisfies \eqref{eq:bd}. Finally, \eqref{eq:bas1} follows as an easy consequence of \cite[Lemma 1]{MMY}.
 
\end{proof}

We also need the following lemma for proving Theorem \ref{thm:pl} and Propositon \ref{ppn:rootn}. The proof of the lemma is deferred to the appendix.
\begin{lem}\label{lem:rootn}

Suppose ${\bf X}=(X_1,\cdots,X_n)$ is an observation from Ising model as in \eqref{eq:ising} such that \eqref{eq:bd} and \eqref{eq:non_trivial} holds. If the true parameter is $(\beta_0,B_0)\in \Theta$, then there exists $\delta>0$ such that
$$\limsup_{n\rightarrow\infty}\frac{1}{n}\log \P_{n,\beta_0,B_0}(|\sum_{i=1}^nX_im_i({\bf X})|<n\delta)<0.$$

\end{lem}

\subsection{Proof of Theorem \ref{thm:pl}}

\begin{enumerate}
\item[(a)]
Setting 
\begin{equation}\label{eq:GradLogL}
\widetilde{PL}_n(\beta,B|{\bf x}):=\sum_{i=1}^n\Big(\beta x_im_i({\bf x})+Bx_i-\log\cosh(\beta m_i({\bf x})+B)\Big)
\end{equation}
note that $PL_n(\beta,B|{\bf x})=\nabla \widetilde{PL}_n(\beta,B|{\bf x}).$
Differentiating the function $(\beta,B)\mapsto \widetilde{PL}_n(\beta,B|{\bf x})$ twice we get the  negative Hessian matrix given by
\begin{align}\label{eq:Hessian}
H_n(\beta,B|{\bf x})=\left[ {\begin{array}{cc}
   \sum_{i=1}^nm_i({\bf x})^2\theta_i(\beta, B|{\bf x}) &   \sum_{i=1}^nm_i({\bf x})\theta_i(\beta, B|{\bf x}) \\        \sum_{i=1}^nm_i({\bf x})\theta_i(\beta, B|{\bf x})  &   \sum_{i=1}^n\theta_i(\beta, B|{\bf x})    \end{array} } \right].
   \end{align}
where $\theta_i(\beta, B|{\bf x}):=\textrm{sech}^2(\beta m_i({\bf x})+B) $.
The determinant of the Hessian is given by
\begin{align}
\notag&\Big[\sum_{i=1}^nm_i({\bf x})^2 \theta_i(\beta, B|{\bf x})\Big]\times \Big[\sum_{i=1}^n \theta_i(\beta, B|{\bf x})\Big]-\Big[\sum_{i=1}^nm_i({\bf x})\theta_i(\beta, B|{\bf x})\Big]^2\\
\notag=&\frac{1}{2}\sum_{i,j=1}^n\theta_i(\beta, B|{\bf x})\theta_i(\beta, B|{\bf x})(m_i({\bf x})-m_j({\bf x}))^2\\
\notag\ge &\frac{1}{2}\text{sech}^4(\beta\gamma+|B|)\sum_{i,j=1}^n(m_i({\bf x})-m_j({\bf x}))^2 =\text{sech}^4(\beta\gamma+|B|)n^2T_n({\bf x}),
\end{align}
which gives
\begin{align}\label{eq:sum2}
|H_n(\beta,B|{\bf x})|=\lambda_n(\beta,B|{\bf X})\mu_n({\bf x})\ge  \text{sech}^4(\beta\gamma+|B|)n^2T_n({\bf x}).
\end{align}
Since on $A_{1,n}^c$ we have $T_n({\bf x})> 0$ it follows that the Hessian is negative definite, and so the function $\widetilde{PL}_n(\beta,B|{\bf x})$ is strictly concave. To show that there exists a global maximizer $(\hat{\beta}_n,\hat{B}_n)$, it thus suffices to show that
$$\lim_{\beta\rightarrow+\pm \infty} \widetilde{PL}_n(\beta,B|{\bf x})=-\infty,\quad \lim_{B\rightarrow\pm\infty}\widetilde{PL}_n(\beta,B|{\bf x})=-\infty.$$
To see this, note that ${\bf x}\in A_{2,n}^c$ implies there exists $i\in [n]$ such that $x_im_i({\bf x})=-|m_i({\bf x})|$, and $m_i({\bf x})\ne 0$. Since we have
$$ \widetilde{PL}_n(\beta,B|{\bf x})\le \beta x_im_i({\bf x})-\log (e^{\beta m_i({\bf x})x_i+B}+e^{-\beta m_i({\bf x})+B}),$$
on letting $\beta\rightarrow\infty$ gives $\lim_{\beta\rightarrow+\infty} \widetilde{PL}_n(\beta,B|{\bf x})=-\infty.$ A similar argument shows that if ${\bf x}\in A_{3,n}^c$, then 
$\lim_{\beta\rightarrow-\infty} \widetilde{PL}_n(\beta,B|{\bf x})=-\infty.$
Finally, it is immediate that
$\lim_{B\rightarrow\pm \infty} \widetilde{PL}_n(\beta,B|{\bf x})=-\infty,$
for any ${\bf x}\notin A_{4,n}$. 
Thus there exists a unique global maximum $(\hat{\beta}_n,\hat{B}_n)$ for the function $(\beta,B)\mapsto \widetilde{PL}_n(\beta,B|{\bf x})\in \R^2$, and so $(\hat{\beta}_n,\hat{B}_n)$ is the unique root of $PL_n(\beta,B|{\bf x})$.
\\

We will now show that if ${\bf x}\in A_{j,n}$ for some $j=1,2,3,4$, then the pseudo-likelihood estimator is not defined. 

\begin{itemize}
\item{${\bf x}\in A_{1,n}$}

On this set we have $T_n({\bf x})=0$ which implies $m_i({\bf x})=\bar{m}({\bf x})$ for all $i\in [n]$. This implies that $Q_n(\beta,B|{\bf x})=\bar{m}({\bf x})R_n(\beta,B|{\bf x})$, and so the equation $PL_n(\beta,B|{\bf x})=(0,0)$ is equivalent to $$R_n(\beta,B|{\bf x})=0\Leftrightarrow \bar{{\bf x}}=\tanh(\beta \bar{m}({\bf x})+B).$$Since the function $(\beta,B)\mapsto\widetilde{PL}_n(\beta,B|{\bf x})$ is convex, it follows that any $(\beta,B)$ satisfying this equation is a global maximizer, and hence in this case the set of maximizers is a line in the two dimensional plane and hence not unique. Thus the pseudo-likelihood estimator is not defined.

\item{${\bf x}\in A_{2,n}$}

On this set we have 
\begin{align*}
Q_n(\beta,B|{\bf x})
=\sum_{i=1}^n|m_i({\bf x})|-\sum_{i=1}^nm_i({\bf x})\tanh(\beta m_i({\bf x})+B)>0,
\end{align*}
and so the equation $Q_n(\beta,B|{\bf x})=0$ has no roots in $\R^2$, and so the pseudo-likelihood estimator is not defined.

\item{${\bf x}\in A_{3,n}$}

Similarly, on this set $Q_n(\beta,B|{\bf x})<0$ for all $(\beta,B)\in \R^2$, and so the pseudo-likelihood estimator is not defined.

\item{${\bf x}\in A_{4,n}$}

If ${\bf x}={\bf 1}$, then we have
\begin{align*}
R_n(\beta,B|{\bf x})
=\sum_{i=1}^n(1-\tanh(\beta m_i({\bf x})+B))>0,
\end{align*}
and so the equation $R_n(\beta,B|{\bf x})=0$ has no roots in $\R^2$, and so the pseudo-likelihood estimator is not defined.

Similarly if ${\bf x}=-{\bf 1}$, then $R_n(\beta,B|{\bf x})<0$ for all $(\beta,B)\in \R^2$.
\\

\end{itemize}

\item[(b)]

Note that if ${\bf x}\in A_{2,n}$ we have
$$Q_n(\beta_0,B_0|{\bf x})\ge (1-\tanh(\beta_0\gamma+|B_0|)\sum_{i=1}^n |m_i({\bf x})|,$$
which gives
$$\Big|\sum_{i=1}^nx_im_i({\bf x})\Big|\le \sum_{i=1}^n |m_i({\bf x})|\le \frac{1}{1-\tanh(\beta_0\gamma+|B_0|)} Q_n(\beta_0,B_0|{\bf x}).$$
Since $Q_n(\beta_0,B_0|{\bf X})=O_p(\sqrt{n})$ by \eqref{eq:bas2} and $\sum_{i=1}^n X_im_i({\bf X})$ is not $o_p(n)$ by Lemma \ref{lem:rootn},  $\P_{n,\beta_0,B_0}(A_{2,n})$ 
 converges to $0$.
A similar proof takes care of $A_{3,n}$. It thus remains to show that $\P_{n,\beta_0,B_0}(A_{4,n})$ converges to $0$ as well. To this effect note that if ${\bf x}=\pm {\bf 1}$ then we have
$|R_n(\beta_0,B_0|{\bf x})|\ge n(1-\tanh(\beta_0\gamma+|B_0|),$ the probability of which converges to $0$ as  $R_{n}(\beta_0,B_0|{\bf X})=O_p(\sqrt{n})$ by \eqref{eq:bas1}.
\\

\item[(c)]
By part (b) we have ${\bf x}\in A_{2,n}^c\cap A_{3,n}^c\cap A_{4,n}^c$ with probability tending to $1$. Also by assumption we have $T_n({\bf X})>0$ with probability tending to $1$, and so the pseudo-likelihood estimator $(\hat{\beta}_n,\hat{B}_n)$ is well defined with probability tending to $1$.
  Recall the $2\times 2$ matrix $H_n(\beta,B|{\bf x})$ as defined in \eqref{eq:Hessian}, and denote $\lambda_n(\beta,B|{\bf x})\ge \mu_n(\beta,B|{\bf x})$ to be its eigenvalues.
   We start by giving a lower bound to the minimum eigenvalue $\mu_n(\beta,B|{\bf x})$. To this effect, note that
   \begin{align*}
\lambda_n(\beta,B|{\bf x})+\mu_n(\beta,B|{\bf x})=\text{tr}(H_n(\beta,B|{\bf x})=\sum_{i=1}^n\theta_i(\beta,B|{\bf x})(m_i^2({\bf x})+1)\le n(1+\gamma^2), 
\end{align*}
which along with \eqref{eq:sum2} gives
\begin{align}\label{eq:eigen_lower}
\mu_n(\beta,B|{\bf x})\ge \frac{\lambda_n(\beta,B|{\bf x})\mu_n(\beta,B|{\bf x})}{\lambda_n(\beta,B|{\bf x})+\mu_n(\beta,B|{\bf x})}=\frac{|H_n(\beta,B|{\bf x})|}{\text{tr}(H_n(\beta,B|{\bf x})}\ge \frac{\text{sech}^4(\beta \gamma+|B|)}{1+\gamma^2}nT_n({\bf x}).
\end{align}
Armed with this estimate, we now complete the proof of the Theorem. To this effect, setting $(\beta_t,B_t)= (t\hat{\beta}_n+(1-t)\beta_0, t\hat{B}_n+(1-t)B_0)$, define a function $g_n:[0,1]\to \RR$ by
 \begin{align*}
g_n(t):= (\hat{\beta}_n-\beta_0)Q_n\left(\beta_t,B_t |{\bf x}\right)+(\hat{B}_n-B_0)R_n\left(\beta_t, B_t|{\bf x}\right),
 \end{align*}
  and note that
  \begin{equation}\label{eq:UpB}
 |g_n(1)-g_n(0)|=|(\hat{\beta}_n-\beta_0)Q_n(\beta_0,B_0|{\bf x})+(\hat{B}_n-B_0)R_n(\beta_0,B_0|{\bf x})|
  =O_p( \sqrt{n} Y_n),
 \end{equation}  
   where $Y_n:=||\hat{\beta}_n-\beta_0,\hat{B}_n-B_0||_2$, and we use Cauchy-Schwarz inequality along with \eqref{eq:bas2} and \eqref{eq:bas1} of Lemma \ref{lem:basak2}. Also we have
   \[g^{\prime}_n(t)= (\hat{\beta}_n-\beta_0,\hat{B}_n-B_0)H_n(\beta_t, B_t|{\bf x})(\hat{\beta}_n-\beta_0,\hat{B}_n-B_0)^{\top}\geq \mu_n(\beta_t,B_t|{\bf x})Y_n^2,\]
 In particular we have $g^{\prime}_n(t)\geq 0$ for all $t\in (0,1)$. Further, using \eqref{eq:eigen_lower} we get the existence of  $r,s>0$ such that
$$\inf_{(\beta,B)\in \Theta:||\beta-\beta_0,B-B_0||\le r}\mu_n(\beta,B|{\bf x})\ge  snT_n({\bf x}).$$
Noting that $||\beta_t-\beta_0,B_t-B_0||_2=tY_n$ gives
\begin{align}\label{eq:upgrade}
\int_0^1 g_n'(t)dt\ge  \int_0^{\min(1,\frac{r}{Y_n})}g_n'(t)dt\ge \min\Big(1,\frac{r}{Y_n}\Big)s n T_n({\bf x})W_n^2,
\end{align}
which along with \eqref{eq:UpB} gives
$\min(Y_n,r)=O_p(\frac{1}{\sqrt{n}T_n({\bf X}})$. Since $r>0$ is fixed, it follows that $Y_n=o_p(1)$, and so $(\hat{\beta}_n,\hat{B}_n)$ converges in probability to $(\beta_0,B_0)$. This shows that $Y_n<r$ with probability tending to $1$, which on using \eqref{eq:upgrade} gives
$\int_0^1 g_n'(t)dt\ge s nT_n({\bf x})Y_n^2$. Along with \eqref{eq:UpB} this gives
$nT_n({\bf X})Y_n=O_p(\sqrt{n})$, which is the claimed bound.

%
%

\end{enumerate}

%
%
%
%
%



\section{Proofs of Theorem \ref{thm:basak2} and Theorem \ref{basak}}\label{sec:three}

The main tool required for proving Theorem \ref{thm:basak2} and Theorem \ref{basak} is the following Theorem, which proves a large deviation estimate for Ising models that might be of independent interest. The proof of this theorem is very similar to the proof of \cite[Theorem 1.6]{CD} and \cite[Theorem 1.1]{BasakM}, and is placed in the appendix. See also \cite[Corollary 12]{EG} which proves a similar result for Ising models under identical conditions (i.e. \eqref{eq:bd} and \eqref{eq:mean_field}). The main difference is that \cite{EG} expresses mean field Ising models as mixture of i.i.d. laws, whereas we focus on the behavior of $m({\bf X})$ which is more relevant to the criterion $T_n({\bf X})=\Theta_p(1)$ provided in Corollary \ref{cor:rootn}.
\begin{defn}
For any ${\bf y}\in [-1,1]^n$ define a vector $b({\bf y})\in [-1,1]^n$ by setting  $b_i({\bf y}):=\tanh(\beta m_i({\bf y})+B)\in [-1,1]$. 
\end{defn}
The next result gives a crucial large deviation estimate for the random vector $b({\bf X})$.
\begin{thm}\label{lem:basak}
Suppose ${\bf X}=(X_1,\cdots,X_n)$ is an observation from the Ising model \eqref{eq:ising}, where the coupling matrix $A_n$ satisfies \eqref{eq:bd} and \eqref{eq:mean_field}.

\begin{enumerate}
\item[(a)]
Let $ r_n:=\sup_{{\bf y}\in [-1,1]^n}\{f_n({\bf y})-I({\bf y})\}$ where $$f_n({\bf y}):=\frac{\beta}{2}{\bf y}'A_n{\bf y}+B\sum_{i=1}^ny_i,\quad I({\bf y}):=\sum_{i=1}^n\Big\{\frac{1+y_i}{2}\log\frac{1+y_i}{2}+\frac{1-y_i}{2}\log \frac{1-y_i}{2}\Big\}.$$ Then we have
$$f_n(b({\bf X}))-I(b({\bf X}))-r_n=o_p(n).$$

\item[(b)]

$$||\nabla f_n(b({\bf X}))-\nabla I (b({\bf X}))||=o_p(\sqrt{n}).$$
\end{enumerate}
\end{thm}



\subsection{Proof of Theorem \ref{thm:basak2}}

Since $T_n=O_p(1)$, it suffices to show that $\frac{1}{T_n}=O_p(1)$, which is equivalent to showing that for any sequence $\{\varepsilon_n\}_{n\ge 1}$ converging to $0$ we have
\begin{align}\label{eq:suffice1}
\lim_{n\rightarrow\infty}\P_{n,\beta,B}(\sum_{i=1}^n(m_i({\bf X})-\bar{m}({\bf X}))^2\le n\varepsilon_n)=0.
\end{align}
Since the set of probability measures on $[-\gamma,\gamma]$ is compact with respect to weak topology, without loss of generality by passing to subsequence we can assume the sequence of empirical measures $\frac{1}{n}\sum_{i=1}^n\delta_{\cR_n(i)}$ converge weakly to $\mu$, where $\mu$ is a probability measure on $[-\gamma,\gamma]$. This along with Dominated Convergence Theorem and  \eqref{eq:irregular} gives
\begin{align}\label{eq:ir2}
\lim_{n\rightarrow\infty}\frac{1}{n}\sum_{i=1}^n(\cR_n(i)-\bar{\cR}_n)^2=\int_{[-\gamma,\gamma]}(\theta-\E_\mu\theta)^2d\mu(\theta)>0.
\end{align}

Proceeding to show \eqref{eq:suffice1}, first note that
\begin{align}\label{eq:not1}
\max_{i\in [n]}|b_i({\bf x})|=\max_{i\in [n]}|\tanh(\beta m_i({\bf x})+B)|\le \tanh(\beta \gamma+|B|)=:p<1,
\end{align}
and so $b_i({\bf x})\in [-p,p]$. Also use \eqref{eq:bd} to note that there exists a finite positive constant $C(\beta,B,\gamma)$ such that
\begin{align*}
\sum_{i=1}^n(m_i({\bf x})-\bar{m}({\bf x}))^2=&\frac{1}{2n}\sum_{i,j=1}^n(m_i({\bf x})-m_j({\bf x}))^2\\
\ge &\frac{C(\beta,B,\gamma)}{2n} \sum_{i,j=1}^n(b_i({\bf x})-b_j({\bf x}))^2\\
=&C(\beta,B,\gamma)\sum_{i=1}^n(b_i({\bf x})-\bar{b}({\bf x}))^2,
\end{align*}
and so changing variables to $\delta_n:=\varepsilon_n/C(\beta,B,\gamma)$ for verifying \eqref{eq:suffice1} 
it suffices to check that
\begin{align}\label{eq:suffice2}
\lim_{n\rightarrow\infty}\P_{n,\beta,B}(\sum_{i=1}^n(b_i({\bf X})-\bar{b}({\bf X}))^2\le n\delta_n,\max_{i\in [n]}|b_i({\bf X})|\le p)=0.
\end{align}
To show this, note that $$\nabla f_n({\bf y})- \nabla I({\bf y})=\beta A_n{\bf y}+B{\bf 1}-\arctanh({\bf y}).$$
Thus if ${\bf y}$ is such that $\sum_{i=1}^n(y_i-\bar{y})^2\le n\delta_n$ and $\max_{i\in [n]}|y_i|\le p$, then with $\tilde{\bf y}:=\bar{y}{\bf 1}$ Triangle inequality gives
\begin{align*}
 ||\nabla f_n({{\bf y})}-\nabla I({{\bf y}}) ||\ge & ||\nabla f_n({{\bf \tilde{y}})}-\nabla I(\tilde{{\bf y}}) ||-||{\bf y}-{\bf \tilde{y}}||\Big(\beta ||A_n||_2 +\frac{1}{1-p^2}\Big)\\
\ge&||\nabla f_n({\tilde{\bf y})}-\nabla I({\tilde{\bf y}}) ||-\sqrt{n\delta_n} \Big(\beta \gamma +\frac{1}{1-p^2}\Big)\\
=&||\nabla f_n({\tilde{\bf y})}-\nabla I({\tilde{\bf y}}) ||-o(\sqrt{n}).
\end{align*}
Finally, we have
\begin{align*}
||\nabla f_n({{\bf \tilde{y}})}-\nabla I(\tilde{{\bf y}})||^2=&\sum_{i=1}^n\Big(\beta \bar{y}\cR_n(i)+B-\arctanh(\bar{y})\Big)^2\\
\ge & \inf_{t\in [-p,p]} \sum_{i=1}^n (\beta t\cR_n(i)+B-\arctanh(t))^2\\
=&n \inf_{t\in [-p,p]} \int_{-\gamma}^\gamma (\beta t\theta+B-\arctanh(t))^2 d\mu(\theta)+o(n).
\end{align*}
Combining these estimates, on the set $\sum_{i=1}^n(b_i({\bf x})-\bar{b}({\bf x}))^2\le n\delta_n$ we have
\begin{align*}
\frac{1}{\sqrt{n}}||f_n(b({\bf x}))-I(b({\bf x}))||
\ge \sqrt{ \inf_{t\in [-p,p]} \int_{-\gamma}^\gamma (\beta t\theta+B-\arctanh(t))^2 d\mu(\theta)}-o(1),
\end{align*}
from which the desired conclusion follows via part (b) of Lemma \ref{lem:basak}, if we can show that $\inf_{t\in [-p,p]}\int_{-\gamma}^\gamma (\beta t\theta+B-\arctanh(t))^2d\mu(\theta)>0$. If not, then there exists $t\in [-p,p]$ such that  $\int_{-\gamma}^\gamma (\beta t\theta+B-\arctanh(t))^2d\mu(\theta)=0$. If $t=0$ then we have
$$0=\int_{-\gamma}^\gamma (\beta t\theta+B-\arctanh(t))^2d\mu(\theta)=B^2\int_{-\gamma}^\gamma \mu(d\theta)=B^2\ne 0,$$
a contradiction. Finally if $t\ne 0$, then we have
$\theta\stackrel{a.s.}{=}\frac{\arctanh(t)-B}{\beta t}$ is a degenerate random variable, a contradiction to \eqref{eq:ir2}. This completes the proof of the Theorem.

\subsection{Proof of Theorem \ref{basak}}
For proving Theorem \ref{basak} we need the following Lemma, the proof of which follows by simple analysis and can be found for e.g. in \cite[Page 10]{dembo_montanari}).
\begin{lem}\label{lem:brazil}
Fix $(\beta,B)\in\Theta$, and define the function $$\phi(y):=\frac{\beta}{2} y^2+By-I(y),
\quad y\in [-1,1].$$ Then the function $\phi(.)$ has a unique global maximum at some $m_0\in (-1,1)$.
\end{lem}

\begin{proof}[Proof of Theorem \ref{basak}]
Fixing $\varepsilon>0$, it suffices to show that 
$$\lim_{n\rightarrow\infty}\P_{n,\beta,B}(\sum_{i=1}^n(m_i({\bf x})-\bar{m}({\bf x}))^2>n\varepsilon)=0.$$
To this effect, note that $\bar{\cR}_n$ is a bounded sequence of real numbers by \eqref{eq:bd}, and so without loss of generality by passing to a subsequence, we can assume that $\bar{\cR}_n$ converges to $\theta$, say. Note that this also gives 
\begin{align}\label{eq:theta}
\sum_{i=1}^n(\cR_n(i)-\theta)^2=o(n).
\end{align}
Now, use \eqref{eq:bd} to note that there exists a finite positive constant $C(\beta,B,\gamma)$ such that
\begin{align*}
\sum_{i=1}^n(m_i({\bf x})-\bar{m}({\bf x}))^2=&\frac{1}{2n}\sum_{i,j=1}^n(m_i({\bf x})-m_j({\bf x}))^2\\
\le &\frac{C(\beta,B,\gamma)}{2n} \sum_{i,j=1}^n(b_i({\bf x})-b_j({\bf x}))^2\\
=&C(\beta,B,\gamma)\sum_{i=1}^n(b_i({\bf x})-\bar{b}({\bf x}))^2,
\end{align*}
and so with $\delta:=\varepsilon/C(\beta,B,\gamma)$ it suffices to check that
$$\lim_{n\rightarrow\infty}\P_{n,\beta,B}(\sum_{i=1}^n(b_i({\bf x})-\bar{b}({\bf x}))^2>n\delta)=0.$$
Let ${\bf y}\in [-1,1]^n$ be any vector such that $\sum_{i=1}^n(y_i-\bar{y})^2\ge n\delta$, and define a matrix $A_n^{(t)}$ by setting 
\begin{align*}
A_n^{(t)}(i,j):=&A_n(i,j)\text{ if }\max(|\cR_n(i)-\theta|,|\cR_n(j)-\theta|)\le t,\\
=& 0\text{ otherwise.}
\end{align*}
 Then we have
\begin{align}
\notag\sum_{i,j=1}^nA_n(i,j)y_iy_j\le& \sum_{i,j=1}^nA_n^{(t)}(i,j)y_iy_j+2\sum_{i:|\cR_n(i)-\theta|>t}\sum_{j=1}^nA_n(i,j)y_iy_j\\
\notag\le &\sum_{i,j=1}^nA_n^{(t)}(i,j)y_iy_j+2\gamma|\{i\in [n]:|\cR_n(i)-\theta|>t\}\\
\label{eq:t_bound}= &\sum_{i,j=1}^nA_n^{(t)}(i,j)y_iy_j+o(n),
\end{align}
where the last equality uses \eqref{eq:theta}.
Since $\sum_{j=1}^nA_n^{(t)}(i,j)\le \theta+t$, it follows all all eigenvalues of $A_n^{(t)}$ are bounded above by $\theta+t$, and so 
$$\sum_{i,j=1}^nA_n^{(t)}y_iy_j={\bf y}'A_n^{(t)}{\bf y}\le (\theta+t)\sum_{i=1}^ny_i^2,$$ which along with  \eqref{eq:t_bound} gives
\begin{align}
\notag f_n({\bf y})-I({\bf y})=&\frac{\beta}{2}{\bf y}'A_n{\bf y}+B\sum_{i=1}^ny_i-\sum_{i=1}^nI(y_i)\\
\notag\le &\frac{\beta}{2}(\theta+t)\sum_{i=1}^ny_i^2+B\sum_{i=1}^ny_i-\sum_{i=1}^nI(y_i)+o(n)\\
\le &n\beta t+\sum_{i=1}^n \phi(y_i)+o(n)\label{eq:brazil}
\end{align}
where $\phi(y):=\frac{\beta\theta}{2} y^2+By-I(y)$. By Lemma \ref{lem:brazil} it follows that $\phi(.)$ has a unique global maximum in $[-1,1]$ at some point $m_0\in (-1,1)$. Define another function
$\Phi:[-1,1]\mapsto [0,\infty)$ by setting $$\Phi(y):=\frac{\phi(m)-\phi(y)}{(y-m_0)^2}\text{ for }y\ne m_0,$$ and note that $\Phi(.)$ is strictly positive for all $y\in [-1,1]$ other than $m_0$ and satisfies
$$\lim_{y\rightarrow m_0}\Phi(y)=-2\phi''(m_0) >0.$$
Consequently, $\Phi(y)$ extends to a  strictly positive continuous function on $[-1,1]$, and so $\alpha:=\inf_{y\in [-1,1]}\Phi(y)>0$, which in turn implies
$$\phi(y)\le \phi(m_0)-\alpha(y-m_0)^2$$
for all $y\in [-1,1]$.
This, along with \eqref{eq:brazil} gives
\begin{align}\label{eq:1_up}
\sup_{{\bf y}:\sum_{i=1}^n(y_i-\bar{y})^2>n\delta}\{f_n({\bf y})-I({\bf y})\}\le o(n)+n\beta t+n\phi(m_0)-n\alpha \delta,
\end{align}
where the last inequality also uses the fact that $$\sum_{i=1}^n(y_i-m_0)^2\ge \sum_{i=1}^n(y_i-\bar{y})^2\ge n\delta.$$

To complete the proof, restricting the sup over all vector ${\bf y}$ which are constant we get
\begin{align}
\notag\sup_{{\bf y}\in [-1,1]^n}\{f_n({\bf y})-I({\bf y})\}\ge& n\sup_{y\in [-1,1]}\{\frac{\beta}{2n}y^2 {\bf 1}'A_n{\bf 1}+By-I(y)\}\\
\notag=&o(n)+n\sup_{y\in [-1,1]}\{\frac{\beta}{2}\theta y^2+By-I(y)\}\\
\label{eq:1_low}=&o(n)+n\phi(m_0)
\end{align}
where the intermediate step uses \eqref{eq:theta} to note that $${\bf 1}'A_n{\bf 1}=\sum_{i=1}^n\cR_n(i)=n\theta+\sum_{i=1}^n(\cR_n(i)-\theta)=n\theta+o(n).$$ Thus combining \eqref{eq:1_up} and \eqref{eq:1_low} gives
$$\sup_{{\bf y}\in [-1,1]^n}\{f_n({\bf y})-I({\bf y})\}-\sup_{{\bf y}:\sum_{i=1}^n(y_i-\bar{y})^2>n\delta}\{f_n({\bf y})-I({\bf y})\}\ge n\alpha \delta-n\beta t+o(n),$$
from which the desired conclusion follows on using Theorem \ref{lem:basak}, since $t>0$ is arbitrary.

\end{proof}

\section{Proof of Theorem~\ref{thm:cw} and Theorem~\ref{ppn:z1}}\label{sec:four}

\subsection{Proof of Theorem \ref{thm:cw}}
To see why the fact that $\Q_n\times \cG(n,p)$ is contiguous to $\mathbb{P}_{n,\beta,B}^{\mathrm{er}}$ implies non existence of consistent estimators, suppose there exists a consistent estimator $(\tilde{\beta}_n,\tilde{B}_n)$ on $\Theta_t$. Now fixing $(\beta_1,B_1)$ and $(\beta_2,B_2)$ in $\Theta_t$, there exists disjoint open balls $\cB_1, \cB_2$ in $\R^2$ such that $(\beta_i,B_i)\in \cB_i$ for $i=1,2$. Consistency implies 
$$\lim_{n\rightarrow\infty}\P^{\mathrm{er}}_{n,\beta_i,B_i}((\tilde{\beta}_n,\tilde{B}_n)\in \cB_i)=1,$$
which along with contiguity gives
$$\lim_{n\rightarrow\infty}(\Q_n\times \cG(n,p))((\tilde{\beta}_n,\tilde{B}_n)\in \cB_i))=1.$$
But this is a contradiction, as $\cB_1$ and $\cB_2$ are disjoint, thus completing the proof of the Theorem. 
\\

The rest of the proof is broken into two parts. Part (a) shows that the probability sequence $\Q_n$ is contiguous to the Curie Weiss model $\mathbb{P}_{n,\beta,B}^{\mathrm{cw}}$ (Ising model on the complete graph). Part (b) then shows that $\mathbb{P}_{n,\beta,B}^{\mathrm{cw}}\times \cG(n,p)$  is contiguous to $\mathbb{P}_{n,\beta,B}^{\mathrm{er}}$.

\begin{enumerate}
\item[(a)]
 By \cite[Lemma 3]{MMY} we have the existence of a random variable $W_n$ with density proportional to $e^{-nf(w)}$, where $f(w):=\beta w^2/2+Bw -\log\cosh(w)$. Also given $W_n=w$ we have $X_1,\cdots,X_n$ are i.i.d. random variables on $\{-1,1\}$ such that
$$\P^{\mathrm{cw}}_{n,\beta,B}(X_i=1|W_n=w)=\frac{e^{\beta w+B}}{e^{\beta w+B}+e^{-\beta w-B}}=1-\P^{\mathrm{cw}}_{n,\beta,B}(X_i=-1).$$
Finally, the conditional distribution of $W_n$ given $X_1=x_1,\cdots,X_n=x_n$ is $N(\bar{x},\frac{1}{n\beta})$. By a slight abuse of notation we use $\P_{n,\beta,B}^{\mathrm{cw}}$ to denote the joint law of $(X_1,\cdots,X_n)$ and $W_n$ on $\{-1,1\}^n\times\R$. Similarly, extend $\Q_{n}$ to $\{-1,1\}^n\times \R$ by setting $W_n$ to be independent of $(X_1,\cdots,X_n)$ with a density proportional to $e^{-f_n(w)}$. Thus under both $\P^{\mathrm{cw}}_{n,\beta,B}$ and $\Q_n$ the marginal distribution of $W_n$ is the same.

We now show that $\P^{\mathrm{cw}}_{n,\beta,B}$ is contiguous to  $\Q_{n}$. To this effect, using \cite{Ellis_Newman} under $\P^{\mathrm{cw}}_{n,\beta,B}$ we have
$$\sqrt{n}(\bar{X}_n-t)\stackrel{d}{\rightarrow}N\Big(0,\frac{1-t^2}{1-\theta(1-t^2)}\Big).$$
This implies
$$\sqrt{n}(W_n-t)\stackrel{d}{\rightarrow}N\Big(0,\frac{1}{\theta[1-\theta(1-t^2)]}\Big)$$
under both $\P^{\mathrm{cw}}_{n,\beta,B}$ and $\Q_n$, as the marginal law of $W_n$ is the same under both measures.
Using this along with a one term Taylor's expansion gives
$$\P^{\mathrm{cw}}_{n,\beta,B}(X_i=1|W_n)=\frac{1}{1+e^{-2\beta W_n-2B}}=\frac{1}{1+e^{-2\beta t-2B}}+\frac{\widetilde{W}_n}{\sqrt{n}}=\alpha+\frac{\widetilde{W}_n}{\sqrt{n}}$$
where $\widetilde{W}_n:=\xi_n\sqrt{n}(W_n-t)$
for some  bounded random variable $\xi_n$, and $\alpha:=\frac{1}{1+e^{-2\beta t-2B}}$. Since $\sqrt{n}(W_n-t)$ is $O_p(1)$ under $\Q_{n}$, it follows that $\widetilde{W}_n=O_p(1)$ as well. Also, setting $S_n:=|i\in [n]:X_i=1|$ we have
$\frac{S_n-n\alpha }{\sqrt{n}}=O_p(1)$
under $\Q_n$.
On the set $|\widetilde{W}_n|\le K$ and $|S_n-n\alpha |\le K\sqrt{n}$  we have
\begin{align*}
&\log \frac{\Q_n(X_1=x_1,\cdots,X_n=x_n|W_n=w)}{\P^{\mathrm{cw}}_{n,\beta_1,B_1}(X_1=x_1,\cdots,X_n=x_n|W_n=w)}\\
=&-S_n\log \frac{\alpha+\frac{\widetilde{W}_n}{\sqrt{n}}}{\alpha}-(n-S_n)\log \frac{1-\alpha-\frac{\widetilde{W}_n}{\sqrt{n}}}{1-\alpha}\\
=&-S_n\Big[\frac{\widetilde{W}_n}{\alpha\sqrt{n}}+O\Big(\frac{K^2}{n}\Big)\Big]+(n-S_n)\Big[\frac{\widetilde{W}_n}{(1-\alpha)\sqrt{n}}+O\Big(\frac{K^2}{n}\Big)\Big]\\
=&-\frac{\widetilde{W}_n}{\sqrt{n}}\Big[\frac{S_n}{\alpha}-\frac{n-S_n}{1-\alpha}\Big]+O(K^2)\\
=&-\frac{\widetilde{W}_n}{\sqrt{n}}\Big[\frac{S_n-n\alpha}{\alpha(1-\alpha)}\Big]+O(K^2)\\
\le &\frac{K^2}{\alpha(1-\alpha)}+O(K^2)=:z_K.
\end{align*}
Thus if $A_n\subset \{-1,1\}^n$ is any sequence of sets such that $\lim_{n\rightarrow\infty}\P^{\mathrm{cw}}_{n,\beta,B}({\bf X}\in A_n)=0,$ then denoting $C_n:=(\int_\R e^{-f_n(w)}dw)^{-1}$ we have
\begin{align*}
&\Q_{n}({\bf X}\in A_n,|\widetilde{W}_n|\le K,|S_n-np|\le K\sqrt{n})\\
=&C_n\int_\R \Q_n({\bf X}\in A_n,|\widetilde{W}_n|\le K,|S_n-np|\le K\sqrt{n}|W_n=w)e^{-f_n(w)}dw\\
\le & C_nz_K\int_\R \P^{\mathrm{cw}}_{n,\beta,B}(A_n,|\widetilde{W}_n|\le K, |S_n-np|\le K\sqrt{n}|W_n=w)e^{-f_n(w)}dw\\
= &z_K\P^{\mathrm{cw}}_{n,\beta,B}({\bf X}\in A_n,|\widetilde{W}_n|\le K,|S_n-np|\le K\sqrt{n})\\
\le&z_KP^{\mathrm{cw}}_{n,\beta,B}({\bf X}\in A_n).
\end{align*}
This gives
$$\Q_n({\bf X}\in A_n)\le z_K\P^{\mathrm{cw}}_{n,\beta,B}({\bf X}\in A_n)+\Q_n(|\widetilde{W}_n|>K)+\Q_n(|S_n-np|>K\sqrt{n}),$$
which on letting $n\rightarrow\infty$ followed by $K\rightarrow\infty$ gives
$$\limsup_{n\rightarrow\infty}\Q_{n}({\bf X}\in A_n)=0,$$
and so $\Q_n$ is contiguous to $\P^{\mathrm{cw}}_{n,\beta,B}$ and the proof is complete. Even though we don't need it, we note that in this case a symmetric proof gives the reverse conclusion as well, i.e. $\P^{\mathrm{cw}}_{n,\beta,B}$ and $\Q_n$ are mutually contiguous.
\\

\item[(b)]
%
%
%



We now show that $\mathbb{P}^{\mathrm{cw}}_{n,\beta, B}\times \cG(n,p)$ is contiguous to $\mathbb{P}^{\mathrm{er}}_{n,\beta, B}$, for which invoking Proposition~6.1 of \cite{BhattacharyaM} it further suffices to show that
 $D(\mathbb{P}^{\mathrm{cw}}_{n,\beta, B}\times \cG(n,p)||\mathbb{P}^{\mathrm{er}}_{n,\beta, B})=O(1),$ where $D(.||.)$ is the Kullback Leibler divergence. A direct computation shows that $ D( \mathbb{P}^{\mathrm{cw}}_{n,\beta, B}\times\cG(n,p)|| \mathbb{P}^{\mathrm{er}}_{n,\beta, B})$ equals
 \begin{align*}
&
\E_{\cG(n,p)} \sum_{{\bf x}\in \{-1,1\}^n} \mathbb{P}^{\mathrm{cw}}_{n,\beta,B}({\bf x}) \left(\frac{\beta}{n-1}\sum_{i,j=1}^n\left[1-\frac{G_n(i,j)}{p}\right]x_ix_j+\log \frac{Z^{\mathrm{er}}_n(\beta,B)}{Z^{\mathrm{cw}}_n(\beta, B)} \right)\\
= &\mathbb{E}_{\mathcal{G}(n,p)}\log Z^{\mathrm{er}}_n(\beta,B)- \log Z^{\mathrm{cw}}_n(\beta, B)
\le \log \mathbb{E}_{\mathcal{G}(n,p)} Z^{\mathrm{er}}_n(\beta,B)- \log Z^{\mathrm{cw}}_n(\beta, B),
\end{align*}
where the last inequality is by Jensen's inequality, and $Z_n^{\mathrm{cw}}(\beta,B)$ and $Z_n^{\mathrm{er}}(\beta,B)$ denote the normalizing constants for the corresponding Ising models. Finally note that
\begin{align*}
\mathbb{E}_{\mathcal{G}(n,p)}Z^{\mathrm{er}}_n(\beta, B)& = \sum_{{\bf x}\in \{-1,1\}^n} \prod_{1\le i<j\le n}^n \mathbb{E}_{\mathcal{G}(n,p)} e^{\frac{\beta}{(n-1)p}G_n(i,j)x_ix_j +B\sum_{i=1}^nx_i}\\
&\leq  \sum_{{\bf x}\in \{-1,1\}^n} \prod_{1\le i<j\le n} \exp\left(\frac{\beta}{n-1} x_ix_j+ \frac{\beta^2 p(1-p)}{2n^2 p^2}\right) e^{B\sum_{i=1}^nx_i},
\end{align*}
where we use the bound
 $\mathbb{E}e^{t(Bin(1,p)-p)}\leq \exp(t^2p(1-p)/2)$. Combining we have 
\[\mathbb{E}_{\mathcal{G}(n,p)} Z^{\mathrm{er}}_n(\beta, B)\le e^{\frac{\beta^2(1-p)}{4p}}\sum_{\vec{x}\in \{-1,1\}^n} e^{\frac{\beta}{n-1}+B\sum_{i=1}^nx_i}=  e^{\frac{\beta^2(1-p)}{4p}} Z^{\mathrm{cw}}_n(\beta, B),\]
which gives an upper bound of 
$\frac{\beta^2(1-p)}{4p}$ to the Kullback-Leibler divergence, and hence completes the proof of the Theorem.

\end{enumerate}

\begin{remark}
To show a similar impossibility result for general dense regular graphs, we need to get exact upper bounds on $Z_n(\beta,B)$. More precisely, we need to show that $Z_n(\beta,B)\le O(1) Z_n^{\mathrm{cw}}(\beta,B)$. 
\end{remark}

\subsection{Proof of Theorem \ref{ppn:z1}}

For proving Theorem \ref{ppn:z1} we need the following two Lemmas. The proofs of the Lemmas are deferred to the end of the section.
The first lemma gives an estimate similar to Lemma \ref{lem:basak2}.

\begin{lem}\label{lem:sparse_new}
Suppose ${\bf X}=(X_1,\cdots,X_n)$ is an observation from the Ising model \eqref{eq:ising}, where the coupling matrix $A_n$ satisfies \eqref{eq:bd} and \eqref{eq:non_trivial}, and $(\beta,B)\in \Theta$. Then we have
\begin{align*}
\limsup_{n\rightarrow\infty}\frac{1}{n}\E \Big[\sum_{i=1}^n(m_i({\bf x})- \sum_{j=1}^n A_n(i,j)\tanh(\beta m_j({\bf x})+B))\Big]^2 <\infty.
\end{align*}
\end{lem}

The second lemma proves an estimate for the Curie-Weiss model, which is necessary for completing the proof of Theorem \ref{ppn:z1}.

\begin{lem}\label{lem:lipschitz}
Let $\beta_n$ be a sequence of positive reals bounded away from $0$ and $+\infty$, and $A_n$ be a matrix satisfying \eqref{eq:bd} and \eqref{eq:Sparseness}. 
Then for any $B\in \R$ there exists $\delta>0$ such that under the \text{Curie-Weiss }model $\P^{\mathrm{cw}}_{n,\beta_n,B}$ we have
$$\limsup_{n\rightarrow\infty}\frac{1}{n}\log \P^{\mathrm{cw}}_{n,\beta_n,B}(\sum_{i=1}^n(m_i({\bf x})-\bar{m}({\bf x}))^2\le n\delta)<0.$$
where $m({\bf x}):=A_n{\bf x}$.
\end{lem}


\begin{proof}[Proof of Theorem \ref{ppn:z1}]
It suffices to show that given an arbitrary sequence of positive reals $\{\varepsilon_n\}_{n\ge 1}$ we have
\begin{align*}
\lim_{n\rightarrow\infty}\P_{n,\beta,B}(\sum_{i=1}^n(m_i({\bf x})-\bar{m}({\bf x}))^2\le n\varepsilon_{n})=0.
\end{align*}
Assume by way of contradiction that 
\begin{equation*}
\limsup_{n\rightarrow\infty}\P_{n,\beta,B}(E_n)>0,\quad 
E_n:=\left\{\sum_{i=1}^n (m_i({\bf x})-\bar{m}({\bf x}))^2 \leq n\epsilon_n \right\}.
\end{equation*}
Using \eqref{eq:bas1} and Lemma \ref{lem:sparse_new}  and increasing the value of $\varepsilon_n$ if necessary, we have  $\lim_{n\rightarrow\infty}\P_{n,\beta,B}( C_n\cap D_n)=1$, where
\begin{align*}
C_n:=&\Big\{\sum_{i=1}^n(x_i-\tanh(\beta m_i({\bf x})+B))\le n\varepsilon_n\Big\},\\
D_n:=&\Big\{\sum_{i=1}^n(m_i({\bf x})-\sum_{j=1}^nA_n(i,j)\tanh(\beta m_j({\bf x})+B))\le n\varepsilon_n\Big\}.
\end{align*}
Combining this gives $\limsup_{n\rightarrow\infty}\P_{n,\beta,B}( C_n\cap D_n\cap E_n)>0$, and so there is a subsequence along which $\P_{n,\beta,B}( C_n\cap D_n\cap E_n)\ge \delta$ for some $\delta>0$.  Restricting ourself to this subsequence, 
on the set $ C_n\cap D_n\cap E_n$ we have
 \begin{eqnarray*}
 \sum_{i=1}^nx_i&\stackrel{C_n}{=}&\sum_{i=1}^n\tanh(\beta m_i({\bf x})+B)+o(n)\stackrel{E_n}{=}n\tanh(\beta \bar{m}({\bf x}) +B)+o(n),\\
n\bar{\cR}_n\tanh(\beta \bar{m}({\bf x}) +B)&\stackrel{E_n}{=}&\sum_{i=1}^n \sum_{j=1}^n A_n(i,j)\tanh(\beta m_i({\bf x}) +B)  +o(n)\stackrel{D_n}{=}\sum_{i=1}^n m_i({\bf x})+o(n) .
\end{eqnarray*} 
In the above sequence of equations, the $o(n)$ terms are uniform non random bounds over the set $C_n\cap D_n\cap E_n$ which on dividing by $n$ go to $0$ as $n\rightarrow\infty$, and are not made explicit for the sake of clarity.
On combining the above two equations we get $\bar{\cR}_n\sum_{i=1}^nx_i=\sum_{i=1}^nm_i+o(n)$,
using which gives
 $$\sum_{i=1}^nx_im_i({\bf x})\stackrel{E_n}{=}\bar{m}\sum_{i=1}^nx_i=n\bar{\cR}_n\bar{x}^2+o(n).$$
 Thus, when ${\bf x}\in C_n\cap D_n\cap E_n$, we have
 \begin{align*}
\delta\le & \P_{n,\beta,B}( C_n\cap D_n\cap E_n)\\
=&\frac{1}{Z_n(\beta,B)}\sum_{{\bf x}\in  C_n\cap D_n\cap E_n} e^{\frac{\beta}{2}{\bf x}'A_n{\bf x}+B\sum_{i=1}^nx_i}\\
\le &\frac{e^{o(n)}}{Z_n(\beta,B)}\sum_{{\bf x}\in C_n\cap D_n\cap E_n} e^{\frac{\beta\bar{\cR}_n}{2}\bar{x}^2+nB\bar{x}}\\
=& e^{o(n)}\frac{Z_{n}^{\mathrm{cw}}(\beta
_n,B)}{Z_n(\beta,B)}\P_{n,\beta_n,B}^{\mathrm{cw}}( C_n\cap D_n\cap E_n),
\end{align*}
where $\beta_n:=\beta\bar{ \cR}_n$ is a sequence of positive real bounded away from $\infty$ and $0$ by \eqref{eq:bd} and \eqref{eq:Sparseness} respectively, and $\P_{n,\beta_n,B}^{\mathrm{cw}}$ is the Curie-Weiss model with parameters $(\beta_n,B)$.
Now, using \cite[(1.8)]{BasakM} and \cite[(2.3)]{BasakM} we get
\begin{align*}
\log Z_n(\beta,B)\ge &n\sup_{t\in [-1,1]}\Big\{\frac{\beta}{2}\bar{\cR}_nt^2+Bt-I(t)\Big\},\\
\log Z_n^{\mathrm{cw}}(\beta_n,B)=&n\sup_{t\in [-1,1]}\Big\{\frac{\beta}{2}\bar{\cR}_nt^2+Bt-I(t)\Big\}+o(n)
\end{align*}
respectively, which readily gives
$$\delta\ge e^{o(n)} \P_{n,\beta_n,B}^{\mathrm{cw}}( C_n\cap D_n\cap E_n)\le e^{o(n)} \P_{n,\beta_n,B}^{\mathrm{cw}}(E_n),$$
which on taking $\log$, dividing by $n$ and letting $n\rightarrow\infty$ gives
$$\liminf_{n\rightarrow\infty}\frac{1}{n}\log  \P_{n,\beta_n,B}^{\mathrm{cw}}(E_n)=0.$$ But this is a contradiction to Lemma \ref{lem:lipschitz},  which completes the proof of the Theorem.

\end{proof}

\begin{proof}[Proof of Lemma \ref{lem:sparse_new}]
To begin note that
$$\sum_{i=1}^n (m_i({\bf x})- \sum_{j=1}^n A_n(i,j)\tanh(\beta m_j({\bf x})+B))=\sum_{j=1}^n\cR_n(j)((x_j-b_j({\bf x})),$$
where $b_j({\bf x})=\tanh(\beta m_i({\bf x})+B)$ as in Lemma \ref{lem:basak2}. This on squaring and expanding gives
\begin{align}\label{eq:bound_sparse_new}
\notag&\E \Big[\sum_{i=1}^n(m_i({\bf X})- \sum_{j=1}^n A_n(i,j)\tanh(\beta m_j({\bf X})+B))\Big]^2\\
= &\sum_{j=1}^n\cR_n(j)^2\E(X_j-b_j({\bf X}))^2+\sum_{j\ne k}\cR_n(j)\cR_n(k)\E(X_j-b_j({\bf X}))(X_k-b_k({\bf X})),
\end{align}
where the first term in \eqref{eq:bound_sparse_new} is bounded by $n\gamma^2$ by \eqref{eq:bd}. Proceeding to bound the second term, setting $m_j^{(i)}({\bf X})=\sum_{k\ne i}A(j,k)x_k$ we have    
\begin{align*}
&\mathbb{E}(X_i-\tanh(\beta m_i({\bf X})+B))(X_j- \tanh(\beta m^{(i)}_j(\mathbf{X})+B))\\=&\mathbb{E}\left\{\mathbb{E}(X_i-\tanh(\beta m_i({\bf X})+B))\big| \{X_k, k\neq i\})(X_j- \tanh(\beta m^{(i)}_j(\mathbf{X})+B))\right\}
= 0,
\end{align*}
which gives
\begin{align*}
&\Big| \mathbb{E}(X_i-\tanh(\beta m_i({\bf X})+B))(X_j- \tanh(\beta m_j(\mathbf{X})+B))\Big|\\
=&\Big| \mathbb{E}(X_i-\tanh(\beta m_i({\bf X})+B))(\tanh(\beta m_j^{(i)}(\mathbf{X})+B)- \tanh(\beta m_j(\mathbf{X})+B))\Big|\\
\le &\beta A_n(i,j).
\end{align*}
Summing over $i\ne j$, the second term in \eqref{eq:bound_sparse_new} is bounded by 
\begin{align*}
\sum_{i\neq j} \mathcal{R}_n(i)\mathcal{R}_n(j)\beta A_n(i,j) \le n\beta \gamma^3,
\end{align*}
Using \eqref{eq:bound_sparse_new} and combining we get
\begin{align}
\E \Big[\sum_{i=1}^n(m_i({\bf X})- \sum_{j=1}^n A_n(i,j)\tanh(\beta m_j({\bf X})+B))\Big]^2\le n\gamma^2+n\beta \gamma^3.
\end{align}    
This completes the proof of the Lemma. 
\end{proof}

\begin{proof}[Proof of Lemma \ref{lem:lipschitz}]

To begin, use \cite[Lemma 3]{MMY} to note that there exists a random variable $W_n$, such that given $W_n=w$ we have $(X_1,\cdots,X_n)$ are i.i.d. with $$\P^{\mathrm{cw}}_{n,\beta_n,B}(X_i=1)=\frac{e^{\beta_n w+B}}{e^{\beta_n w+B}+e^{-\beta_n w-B}}=1-\P^{\mathrm{cw}}_{n,\beta_n,B}(X_i=-1).$$
Also note that $$Y_n:=\sqrt{\sum_{i=1}^n(m_i({\bf x})-\bar{m}({\bf x}))^2}=\sup_{||{\bf a}||_2\le 1}a_i(m_i({\bf x})-\bar{m}({\bf x}))=\sup_{||{\bf a}||_2\le 1}\sum_{i=1}^n(\tilde{a}_i-\bar{\tilde{{a}}})x_i,$$
where ${\bf\tilde{ a}}:=A_n{\bf a}$. 
Thus invoking \cite[Theorem 12.3]{BLM} gives
      \begin{align}\label{eq:Lowertail}
      \mathbb{P}\left(Y_n\leq \mathbb{E}(Y_n|W_n)-2\sqrt{t\Sigma_n}\Big| W_n\right)\leq 4e^{-\frac{t}{4}},
      \end{align}
      where \[ \Sigma_n:= \mathbb{E}\left[ \sup_{||{\bf a}||\le 1}\sum_{i=1}^n (\tilde{a}_i-\tilde{\bar{a}})^2(X^{\prime}_i -X_i)^2\big|X_1,\cdots,X_n,W_n\right]\] 
       and $(X^{\prime}_1,\ldots ,X^{\prime}_n)$ is an independent and identically distributed copy from the conditional distribution of $(X_1|W_n)$. This result requires the supremum to be over a countable set, but of course we can restrict the supremum to be over all rational tuples in $||{\bf a}||\le1$, which is countable. 
       In order to invoke \eqref{eq:Lowertail} we need to first estimate $\E(Y_n|W_n)$. To this effect, a direct computation gives
       \begin{align*}
       \E (Y_n|W_n)=&\E\Big( \sqrt{(\sum_{i=1}^nm_i({\bf X})-\bar{m}({\bf X}))^2}\Big|W_n\Big)\\
       \ge& \frac{1}{\sqrt{n}} \E\Big( \sum_{i=1}^n|m_i({\bf X})-\bar{m}({\bf X})|\Big|W_n\Big)\text{   [By Cauchy-Schwarz]}\\
       \ge & \frac{1}{2\gamma\sqrt{n}}\E \Big(\sum_{i=1}^n (m_i({\bf X})-\bar{m}({\bf X}))^2\Big|W_n\Big)\text{    [Since $|m_i({\bf x})-\bar{m}({\bf x})|\le 2\gamma$]}\\
       \ge &\frac{1}{2\gamma\sqrt{n}}\sum_{i=1}^n Var\Big(m_i({\bf X})-\bar{m}({\bf X})|W_n\Big)\\
       =& \frac{1}{2\gamma\sqrt{n}}Var(X_1|W_n)\sum_{i,j=1}^n\Big(A_n(i,j)-\frac{1}{n}\cR_n(i)\Big)^2.
       \end{align*}
       Finally we have
 \begin{equation*}
\sum_{i,j=1}^n\Big(A_n(i,j)- \frac{1}{n}\cR_n(i)\Big)^2= \sum_{i,j=1}^n A^2_n(i,j) -\frac{1}{n}\sum_{i=1}^n \cR_n(i)^2\geq \sum_{i,j=1}^{n} A^2_n(i,j)- \gamma^2
\end{equation*}
 where the last inequality follows by \eqref{eq:bd}.
Since $Var(X_1|W_n)=\text{sech}^2(\beta_n W_n+B)$,  on the set $|W_n|\le 2$ we have
$$\E(Y_n|W_n)\ge \frac{1}{2\gamma\sqrt{n}}\text{sech}^2(\beta_n W_n+B)\Big(\sum_{i,j=1}^{n} A^2_n(i,j)- \gamma\Big)\ge 2c\sqrt{n}$$
for some constant $c>0$ by invoking \eqref{eq:Sparseness}.
 Also we have $$\Sigma_n\le 4\sup_{||{\bf a}||_2\le 1}\sum_{i=1}^n\tilde{a}_i^2\le 4\gamma^2,$$ and so on the set $|W_n|\le 2$ invoking \eqref{eq:Lowertail} with $t=\frac{nc^2}{16\gamma^2}$ gives $$\mathbb{P}\left(Y_n\leq c\sqrt{n}| W_n\right)\le \mathbb{P}\left(Y_n\leq 2c\sqrt{n}-4\gamma\sqrt{t}\Big| W_n\right)\leq 4e^{-\frac{t}{4}}=4 e^{-\frac{nc^2}{64\gamma^2}}.$$
With $\delta=c^2$ this gives
\begin{align*}
&\P^{\mathrm{cw}}_{n,\beta_n,B}(\sum_{i=1}^n(m_i({\bf X})-\bar{m}({\bf X}))^2\le n \delta)\\
\le &\P^{\mathrm{cw}}_{n,\beta_n,B}(|W_n|>2)+\E^{\mathrm{cw}}_{n,\beta_n,B}\Big(\P^{\mathrm{cw}}_{n,\beta_n,B}(Y_n\le c\sqrt{n}\Big|W_n)1\{|W_n|\le 2\}\Big)\\
\le &\P^{\mathrm{cw}}_{n,\beta_n,B}(|W_n|>2)+4 e^{-\frac{nc^2}{64\gamma^2}},
 \end{align*}
       from which the desired conclusion follows on using \cite[Lemma 3]{MMY} to note that the first term in the RHS above decays exponentially, as $(W_n|X_1,\cdots,X_n)\sim N(\bar{X},\frac{1}{n\beta_n})$.
       
       \end{proof}

\section{Appendix}\label{sec:appen}
\subsection{Proof of Theorem \ref{lem:basak}}

\begin{enumerate}
\item[(a)]

The proof idea of this theorem is a modification of the proof of \cite[Theorem 1.1]{BasakM} (see also \cite[Theorem 1.6]{CD}) with obvious modifications. Since a central estimate necessary for this paper appears as by product of the proof technique, the proof is given in detail. 
\\

To begin fix $\delta>0$ and note that it suffices to show
$$\limsup_{n\rightarrow\infty}\P_{n,\beta,B}(f_n(b({\bf X}))-I(b({\bf X}))\le r_n-n\delta)=0.$$

Without loss of generality assume $\gamma=1$ for simplicity of exposition, where $\gamma$ is as in \eqref{eq:bd}. Fixing $\varepsilon>0$, by Lemma \ref{lem:basak2} and Markov's inequality we have 
$\lim_{n\rightarrow\infty}\P(C_n)=1,$ where $C_n$ is the intersection of the following three sets:
\begin{align*}
&\Big\{{\bf x}\in \{-1,1\}^n:\Big|\sum_{i=1}^n(x_i-b_i({\bf x}))\Big|\le n\varepsilon\Big\},\\
 &\Big\{{\bf x}\in \{-1,1\}^n:\Big|\sum_{i=1}^n(x_i-b_i({\bf x}))m_i({\bf x})\Big|\le n\varepsilon\Big\},\\
&\Big\{{\bf x}\in \{-1,1\}^n:\Big|f_n({\bf x})-f_n(b({\bf x}))\Big|\le n\varepsilon\Big\}.
 \end{align*}
Also, for ${\bf z},{\bf w}\in [-1,1]^n$ set  $$g_n({\bf z},{\bf w}):=\sum_{i=1}^n\Big\{\frac{1+z_i}{2}\log \frac{1+w_i}{2}+\frac{1-z_i}{2}\log \frac{1-w_i}{2}\Big\},$$
and note that on $C_n$ we have
\begin{align}\label{eq:basak1}
|g({\bf x},b({\bf x}))-I(b({\bf x}))|
=\Big|\sum_{i=1}^n(x_i-b_i({\bf x}))(\beta m_i({\bf x})+B)\Big|\le (\beta+|B|) n \varepsilon.
\end{align}
Thus, setting  $${S}_n:=\{{\bf x}\in \{-1,1\}^n: f_n(b({\bf x}))-I(b({\bf x}))\le r_n-n\delta\},$$
we have
\begin{align}\label{eq:prob}
\P_{n,\beta,B}(S_n)\le &\frac{\sum_{{\bf x}\in S_n\cap C_n}e^{f_n({\bf x})}}{\sum_{{\bf x}\in \{-1,1\}^n}e^{f_n({\bf x})}}+\P_{n,\beta,B}(C_n^c).
\end{align}
For estimating the denominator in the RHS of \eqref{eq:prob}, using the bound of \cite[Theorem 1.6]{CD} we have
\begin{align}\label{eq:denom}
\sum_{{\bf x}\in \{-1,1\}^n} e^{f_n({\bf x})}\ge \sup_{{\bf y}\in [-1,1]^n}e^{f_n({\bf y})-I({\bf y})}=e^{r_n}.
\end{align}
Proceeding to bound the numerator in the RHS of \eqref{eq:prob}, let $N_n(\varepsilon):=\{i\in [n]:|\lambda_i(A_n)|>\varepsilon/\sqrt{2}\}$ and use \eqref{eq:mean_field} to note that
$$\frac{|N_n(\varepsilon)|}{n}\le \frac{2}{n\varepsilon^2}\sum_{i\in [n]}\lambda_i(A_n)^2=\frac{2}{n\varepsilon^2}\sum_{i,j=1}^nA_n(i,j)^2\stackrel{n\rightarrow\infty}{\rightarrow}0.$$
Set $k:=|N_n(\varepsilon)|$, and let $\cD_{n,0}(\varepsilon)$ be a $\varepsilon\sqrt{n/2}$ net of the set $\{{\bf f}\in \R^{k}:\sum_{i=1}^kf_i^2 \le n\}$ of size at most $\Big(\frac{3\sqrt{2}}{\varepsilon}\Big)^{k}$. The existence of such a net is standard, see for example \cite[Lemma 2.6]{MS}. Thus letting $\{{\bf p}_1,\cdots,{\bf p}_n\}$ denote the eigenvectors of $A_n$ and setting
$${\cD}_{n,1}(\varepsilon):=\Big\{\sum_{i=1}^kc_i\lambda_i(A_n){\bf p}_i, {\bf c}\in \cD_{n,0}(\varepsilon)\Big\},$$ we claim that ${\cD}_{n,1}(\varepsilon)$ is a $\varepsilon\sqrt{n}$ net of the set $\{A_n{\bf x}:{\bf x}\in \{-1,1\}^n\}$. Indeed, any ${\bf x}\in \{-1,1\}^n$ can be written as $\sum_{i=1}^nf_i{\bf p}_i$, where $\sum_{i=1}^nf_i^2=\sum_{i=1}^nx_i^2=n $. In particular this means that $\sum_{i=1}^kf_i^2\le n$, and so there exists ${\bf c}\in \cD_{n,0}(\varepsilon)$ such that $||{\bf c}-{\bf f}||\le \varepsilon\sqrt{n/2}$. Thus with $\sum_{i=1}^kc_i \lambda_i(A_n){\bf p}_i\in {\cD}_{n,1}(\varepsilon)$ we have
\begin{align*}
||A_n{\bf x}-\sum_{i=1}^k c_i\lambda_i(A_n){\bf p}_i||^2=&\sum_{i=1}^k(c_i-f_i)^2\lambda_i(A_n)^2+\sum_{i=k+1}^n\lambda_i(A_n)^2f_i^2\\
\le& \frac{1}{2}n\varepsilon^2+\frac{1}{2}n \varepsilon^2=n\varepsilon^2,
\end{align*}
where we have used \eqref{eq:bd} to note that $$\max_{i\in [n]}|\lambda_i(A_n)|\le \max_{i\in [n]}\sum_{j=1}^n |A_n(i,j)|\le 1.$$
In particular for any ${\bf x}\in \{-1,1\}^n$ there exists at least one ${\bf p}\in \cD_{n,1}(\varepsilon)$ such that $||{\bf p}-m({\bf x})||\le \sqrt{n}\varepsilon$.   
For any ${\bf p}\in \cD_{n,1}(\varepsilon)$ let $$\cP({\bf p}):=\{{\bf x}\in \{-1,1\}^n:||{\bf p}-m({\bf x})||\le\varepsilon \sqrt{n}\}.$$ Thus
using \eqref{eq:basak1} we have
\begin{align}
\notag&\sum_{{\bf x}\in {S}_n\cap C_n}e^{f_n({\bf x})}\\
\notag\le &e^{n\varepsilon(1+\beta+|B|)}\sum_{{\bf x}\in {S}_n\cap C_n}e^{f_n(b({\bf x}))-I(b({\bf x}))+g_n({\bf x},b({\bf x}))}\\
\notag\le &e^{n\varepsilon(1+\beta+|B|)}\sup_{{\bf x}\in {S}_n} e^{f_n(b({\bf x}))-I(b({\bf x}))}\sum_{{\bf x}\in \{-1,1\}^n}e^{g_n({\bf x},b({\bf x}))}\\
\label{eq:boundd}\le &e^{n\varepsilon(1+\beta+|B|)+r_n-n\delta}\sum_{{\bf p}\in \cD_{n,1}(\varepsilon)}\sum_{{\bf x}\in \cP({\bf p})}e^{g_n({\bf x},b({\bf x}))}.
\end{align}
Setting ${\bf u}={\bf u}({\bf p}):=\tanh(\beta {\bf p}+B)$,  if $||{\bf p}-m({\bf x})||\le \varepsilon\sqrt{n}$, then we have
\begin{align}
\label{eq:c13}
|g_n({\bf x},b({\bf x}))-g_n({\bf x},{\bf u})|\le 2\sum_{i=1}^n|m_i({\bf x})-{\bf p}_i|\le 2\sqrt{n}||m({\bf x})-{\bf p}||\le 2n\varepsilon.
\end{align}
Combining \eqref{eq:c13} along with \eqref{eq:boundd} gives
\begin{align}
\notag\sum_{{\bf x}\in {S}_n\cap C_n}e^{f_n({\bf x})}\le &e^{n\varepsilon(3+\beta+|B|)+r_n-n\delta}\sum_{{\bf p}\in \cD_{n,1}(\varepsilon)}\sum_{{\bf x}\in \cP({\bf p})}e^{g_n({\bf x},{\bf u}({\bf p}))}\\
\notag\le &e^{n\varepsilon(3+\beta+|B|)+r_n-n\delta}\sum_{{\bf p}\in \cD_{n,1}(\varepsilon)}\sum_{{\bf x}\in \{-1,1\}^n} e^{g_n({\bf x},{\bf u}({\bf p}))}\\
\label{eq:num}=&e^{n\varepsilon(3+\beta+|B|)+r_n-n\delta}|\cD_{n,1}(\varepsilon)|,
\end{align}
where the last equality uses the fact that $\sum_{{\bf x}\in \{-1,1\}^n}e^{g_n({\bf x},{\bf u})}=1$ for any ${\bf u}\in [-1,1]^n$.

Combining \eqref{eq:num} and \eqref{eq:denom} along with \eqref{eq:prob} gives
$$\P_{n,\beta,B}(S_n)\le e^{n\varepsilon(3+\beta+|B|)-n\delta}|\cD_{n,1}(\varepsilon)|+\P_{n,\beta,B}(C_n^c).$$
Since $\lim_{n\rightarrow\infty}\P_{n,\beta,B}(C_n^c)=0$, it suffices to control the first term above. To this effect, recall that $|\cD_{n,1}(\varepsilon)|=|\cD_{n,0}(\varepsilon)|\le \Big(\frac{3\sqrt{2}}{\varepsilon}\Big)^{k_n}$,
and so
$$\frac{1}{n}\log  \Big[e^{n\varepsilon(3+\beta+|B|)-n\delta}|\cD_{n,1}(\varepsilon)|\Big]=\varepsilon(3+\beta+|B|)-\delta ,$$
from which the desired conclusion follows since $\varepsilon>0$ is arbitrary.


\item[(b)]

Fixing $\delta>0$ it suffices to show that
$$\P_{n,\beta,\B}(||\nabla f_n(b({\bf X}))-\nabla I(b({\bf X}))||^2>n\delta)=0.$$
To this effect, use \eqref{eq:bd} to note that $|b_i({\bf x})|\le \tanh(\beta \gamma+|B|)=:p\in (0,1)$. 
Now if ${\bf y}\in [-p,p]^n$ is such that $||\nabla f_n({\bf y})-\nabla I({\bf y})||^2>n\delta$, then setting ${\bf y}^{(t)}:={\bf y}+t(\nabla f_n({\bf y})-\nabla I({\bf y}))$ we claim that ${\bf y}^{(t)}\in [-1,1]^n$ for $t\in (0,\varepsilon)$ for $\varepsilon:=(1-p)(1-p^2)>0$ free of $n$. Indeed this follows on noting that $$\Big|\frac{\partial^2 f_n({\bf y})}{\partial y_i^2}-\frac{\partial^2 I({\bf y})}{\partial y_i^2}\Big|=\frac{1}{1-y_i^2}\le \frac{1}{1-p^2}.$$
But then a two term Taylor expansion gives
\begin{align*}
&\{f_n({\bf y}^{(t)})-I({\bf y}^{(t)})\}- \{f_n({\bf y})-I({\bf y})\}\\
\ge&t||\nabla f_n({\bf y})-\nabla I({\bf y})||^2-\frac{t^2}{2}||\nabla f_n({\bf y})-\nabla I({\bf y})||^2\Big(\beta \gamma+\frac{1}{1-p^2}\Big)\\
\ge  &nt\delta\Big[1-\frac{t}{2}(\beta \gamma+\frac{1}{1-p^2})\Big].
\end{align*}
There exists $t\in (0,\varepsilon)$ such that $1-\frac{t}{2}(\beta \gamma+\frac{1}{1-p^2})>0$. Fixing such a $t$, we have
$$\sup_{{\bf y}:||\nabla f_n({\bf y})-\nabla I({\bf y})||^2>n\delta}\{f_n({\bf y})-I({\bf y})\}\le \sup_{{\bf y}\in [-1,1]^n}\{f_n({\bf y})-I({\bf y})\}-nt\delta\Big[1-\frac{t}{2}(\beta \gamma+\frac{1}{1-p^2})\Big].$$
This along with part (a) gives the desired conclusion.

\end{enumerate}

\subsection{Proof of Lemma \ref{lem:rootn}}
By symmetry it suffices to consider $B_0>0$ and show that
$$\limsup_{n\rightarrow\infty}\frac{1}{n}\log \P_{n,\beta_0,B_0}(\sum_{i=1}^nX_im_i({\bf X})<n\delta)<0.$$
Setting $\lambda:=\beta_0/4$, a direct argument using Markov's inequality gives
\begin{align*}
\P_{n,\beta_0,B_0}(\sum_{i=1}^nX_im_i({\bf X})<n\delta)\le& e^{\lambda n \delta} \E_{n,\beta_0,B_0} e^{-\lambda \sum_{i=1}^nX_im_i({\bf X})}\\
=& e^{\lambda n\delta+Z_n(\beta_0-2\lambda,B_0)-Z_n(\beta_0)}\\
\le &e^{\lambda n\delta-2\lambda Z_n'(\beta_0-2\lambda,B_0)},
\end{align*}
where $Z_n'(\beta,B):=\frac{\partial Z_n(\beta,B)}{\partial \beta}$, and the last inequality uses the fact that the function $\beta\mapsto Z_n'(\beta,B)$ is non decreasing in $\beta$. On taking $\log$ on both sides, dividing by $n$ and letting $n\rightarrow\infty$ gives
$$\limsup_{n\rightarrow\infty}\frac{1}{n}\log  \P_{n,\beta_0,B_0}(\sum_{i=1}^nX_im_i({\bf X})<n\delta)\le \frac{\beta_0}{4}\Big(\delta-\frac{2}{n}Z_n'(\beta_0/2,B_0)\Big),$$
and so it suffices to show that
\begin{align}\label{eq:suffice_m}
\liminf_{n\rightarrow\infty}\frac{2}{n}Z_n'(\beta_0/2,B_0)>0.
\end{align}
To show this note that
\begin{align}
\notag 2Z_n'(\beta_0/2,B_0)=&\E_{n,\beta_0/2,B_0}\sum_{i=1}^nX_im_i({\bf X})\\
\notag=&\E_{n,\beta_0/2,B_0}\sum_{i=1}^nm_i({\bf X})\tanh(\beta_0 m_i({\bf X})+B_0)\\
\notag=&\tanh(B_0)\E_{n,\beta_0/2,B_0}\sum_{i=1}^nm_i({\bf X})+\beta_0\E_{n,\beta_0/2,B_0}\sum_{i=1}^n m_i({\bf X})^2\text{sech}^2(\xi_i)\\
\label{eq:est_new_1}\ge &\tanh(B_0)\E_{n,\beta_0/2,B_0}\sum_{i=1}^nm_i({\bf X}),
\end{align}
where $\xi_i$ is a random variable lying between $B_0$ and $B_0+\beta m_i({\bf X})$. To show this use elementary calculus to note that there exists $C>0$ such that $$\inf_{|x|\le \beta_0\gamma,y>0}\frac{\tanh(x+y)-\tanh(x)}{\tanh(y)}\ge C,$$ which gives
\begin{align}
\notag\E_{n,\beta_0/2,B_0}\sum_{i=1}^nm_i({\bf X})=&\E_{n,\beta_0/2,B_0}\sum_{j=1}^n\cR_j\tanh(\beta_0 m_j({\bf X})+B_0)\\
\notag\ge &\sum_{j=1}^n \cR_j\E_{n,\beta_0/2,B_0}\tanh(\beta_0 m_j({\bf X}))+C\tanh(B_0)\sum_{j=1}^n\cR_j\\
\label{eq:est_new_2}\ge&C\tanh(B_0)\sum_{j=1}^n\cR_j,
\end{align}
where the last inequality uses the fact that the function ${\bf x}\mapsto \beta_0 m_j({\bf x})$ is an increasing function in the sense of \cite[Theorem 2.1]{grimmet}, along with the fact that the measure $\P_{n,\beta_0/2,B_0}$ is stochastically larger than $\P_{n,\beta_0/2,0}$, giving
$$\E_{n,\beta_0/2,B_0}\tanh(\beta_0 m_j({\bf X})\ge \E_{n,\beta_0/2,0}\tanh(\beta_0 m_j({\bf X}))=0.$$ The desired conclusion \eqref{eq:suffice_m} then follows from \eqref{eq:est_new_1} and \eqref{eq:est_new_2}.

\subsection{Proof of Proposition \ref{ppn:rootn}}
\begin{enumerate}
\item[(a)]
To begin, use a one term Taylor expansion to get
\begin{align}\label{eq:beta_est}
0=Q_n(\hat{\beta}_n,B|{\bf X})=Q_n(\beta_0,B|{\bf X})+\int_{\beta_0}^{\hat{\beta}_n}\frac{\partial}{\partial t} Q_n(t,B|{\bf X}) dt,
\end{align}
where 
\begin{align*}
\frac{\partial Q_n(\beta,B|{\bf X})}{\partial \beta}
=-\sum_{i=1}^nm_i({\bf x})^2\text{sech}^2(\beta m_i({\bf x})+B)
\ge -\text{sech}^2(\beta \gamma+|B|)\sum_{i=1}^nm_i({\bf X})^2.
\end{align*}
This along with \eqref{eq:beta_est} gives
$$\sum_{i=1}^nm_i({\bf X})^2 \Big|\int_{\beta_0}^{\hat{\beta}_n} \text{sech}^2(t \gamma+|B|)dt\Big|\le |Q_n(\beta_0,B|{\bf X})|=O_p(\sqrt{n}),$$
where the last step uses \eqref{eq:bas2} along with Markov's inequality. 
The desired conclusion is immediate from this, if we can show
 $\sum_{i=1}^nm_i({\bf X})^2=\Theta_p(n),$ which follows if we can show
\begin{align*}
\lim_{\delta\rightarrow0}\lim_{n\rightarrow\infty}\P_{n,\beta,B}(\frac{1}{n}\sum_{i=1}^nm_i({\bf X})^2>n\delta)=1.
\end{align*}
But this follows on using Cauchy-Schwarz inequality to note that
$$\sum_{i=1}^nm_i({\bf x})^2\ge \frac{1}{n}\Big[\sum_{i=1}^nX_im_i({\bf X})\Big]^2$$
along with Lemma \ref{lem:rootn}.
\item[(b)]

The proof of part (b) follows by similar calculations as in part (a), on using \eqref{eq:bas1} and noting that
$$\frac{\partial R_n(\beta,B)}{\partial B}=-\sum_{i=1}^n\text{sech}^2(\beta m_i({\bf x})+B)\le -n\text{sech}^2(\beta \gamma+|B|).$$

\end{enumerate}

\end{document}